\documentclass[]{amsart}

\usepackage{amsmath}
\usepackage{amsthm}
\usepackage{amsfonts}
\usepackage{amssymb}
\usepackage{todonotes}
\usepackage{mathtools}
\usepackage{csquotes}
\usepackage{graphicx}
\usepackage{faktor}
\usepackage{xfrac}   
\usepackage{tikz-cd}
\usepackage{float}

\usepackage[style=alphabetic,sorting=nyt,backend=bibtex8]{biblatex}
\bibliography{reference.bib}

\usepackage{hyperref}
\hypersetup{colorlinks=true}
\usepackage[shortlabels]{enumitem}

\newcommand{\sqleft}[1]{\mathrel{_{#1}{\sqsubseteq}}}

\newcommand{\MJ}[1]{{#1}_{\scriptscriptstyle\top}}

\usepackage{xcolor}
\usepackage[author=L, final]{fixme}
\fxsetup{targetlayout=color}
\definecolor{fxtarget}{HTML}{007ACC}
\fxsetup{nomargin,nomarginclue,noinline,nofootnote,noindex}

\newtheorem{theorem}{Theorem}[]
\newtheorem{proposition}[theorem]{Proposition}
\newtheorem{corollary}[theorem]{Corollary}
\newtheorem{lemma}[theorem]{Lemma}
\newtheorem{claim}{Claim}[theorem]

\theoremstyle{definition}

\newtheorem{question}{Question}
\newtheorem{ditorproblem}{Ditor's Problem}
\newtheorem{definition}[theorem]{Definition}
\newtheorem*{conjecture*}{Conjecture}

\theoremstyle{remark}
\newtheorem{remark}[theorem]{Remark}

\begin{document}

\title{Ladders and squares}
\author{Lorenzo Notaro}
\address{University of Vienna, Institute of Mathematics, Kurt G\"{o}del Research Center, Kolingasse 14-16, 1090 Vienna, Austria}
\curraddr{}
\thanks{The author would like to acknowledge INDAM for the financial support. This research was partially supported by the project PRIN 2022 “Models, sets and classifications”, prot. 2022TECZJA. This research was funded in whole or in part by the Austrian Science Fund (FWF) \href{https://www.fwf.ac.at/en/research-radar/10.55776/ESP1829225}{10.55776/ESP1829225}. For open access purposes, the author has applied a CC BY public copyright license to any author accepted manuscript version arising from this submission.}
\email{lorenzo.notaro@univie.ac.at}
\subjclass[2020]{Primary 03E05, Secondary 03E35, 06A07}
\keywords{lattice, breadth, lower cover, ladder, semilattice, square principle}

\begin{abstract}
In 1984, Ditor asked two questions:\vspace{0.3em}
\begin{enumerate}
\itemsep0.3em
\item[(A)] For each $n\in\omega$ and infinite cardinal $\kappa$, is there a join-semilattice of breadth $n+1$ and cardinality $\kappa^{+n}$ whose principal ideals have cardinality $< \kappa$?
\item[(B)] For each $n \in \omega$, is there a lower finite lattice of cardinality $\aleph_{n}$ whose elements have at most $n+1$ lower covers?\vspace{0.3em}
\end{enumerate}
We show that both questions have positive answers under the axiom of constructibility, and hence consistently with $\mathsf{ZFC}$. More specifically, we derive the positive answers from assuming that  $\square_\kappa$ holds for enough $\kappa$'s.
\end{abstract}
\maketitle

\section{Introduction}

For a positive integer $n$, an \emph{$n$-ladder} is a lattice whose principal ideals are finite (i.e. a lower finite lattice) and whose elements have at most $n$ lower covers. This family of lattices was originally isolated by Ditor~\cite{MR0732199} and Dobbertin~\cite{MR862871} under different names\footnote{Ditor called these lattices \emph{$n$-lattices} while Dobbertin called them \emph{$n$-frames}. We follow the terminology of Gr\"{a}tzer, Lakser, and Wehrung \cite{MR1768850} and refer the reader to \cite[p. 387]{MR2926318} for a brief history of the name \emph{$n$-ladder}.}. All results due to Ditor that we cite are from \cite{MR0732199}.

Ditor showed that every $n$-ladder has breadth at most $n$, and proved that every join-semilattice of breadth at most $n+1$ whose principal ideals have cardinality $< \kappa$, for some infinite cardinal $\kappa$, has cardinality at most $\kappa^{+n}$ (see Theorem~\ref{thm:Ditor}). Hence, in particular, he proved that every $(n+1)$-ladder, for some $n\in\omega$, has cardinality at most $\aleph_{n}$. He then \fxnote*{}{stated} the question of whether his cardinality bounds are sharp (see also \cite[p. 291]{MR2768581}):\vspace{0.5em}
\textit{
\begin{enumerate}[label={\upshape (\Alph*)}]
\itemsep0.5em
\item\label{questionA} For each $n\in\omega$ and infinite cardinal $\kappa$, is there a join-semilattice of breadth $n+1$ and cardinality $\kappa^{+n}$ whose principal ideals have cardinality $< \kappa$?
\item\label{questionB} For each $n \in \omega$, is there an $(n+1)$-ladder of cardinality $\aleph_n$?\vspace{0.5em}
\end{enumerate}
}

Note that \ref{questionB} can be regarded as a more demanding version of \ref{questionA} \fxnote*{}{in case} $\kappa = \aleph_0$. Indeed, every $(n+1)$-ladder is, in particular, a lower finite join-semilattice of breadth at most $n+1$, but not every lower finite lattice of breadth at most $n+1$ is an $(n+1)$-ladder---e.g. the diamond lattice $\mathsf{M}_3$ has breadth $2$, but it is not a $2$-ladder. 

The case $n=0$ is trivial for both \ref{questionA} and \ref{questionB}. Ditor answered both questions positively in case $n=1$ and $\kappa$ is regular. In particular, he proved the existence of $2$-ladders of cardinality $\aleph_1$. \fxnote*{}{Since then, $2$-ladders have been used primarily in representation problems in universal algebra, particularly for structures of cardinality $\leq \aleph_1$ (e.g. \cite{MR862871,MR1768850, MR1800815, MR2309879}). There are exceptions, however: for instance, in \cite{MR3286148}, $2$-ladders are employed in the description of join-irreducible elements in certain lattices of regular closed sets. By contrast, no significant applications of $n$-ladders with $n \geq 3$ are known, due to category-theoretical obstacles \cite{MR2926318}.

Given his results, Ditor isolated the two simplest instances of \ref{questionA} and \ref{questionB} that were left open in his paper:
\begin{ditorproblem}\label{ditprob1}
Is there a $3$-ladder of cardinality $\aleph_2$?
\end{ditorproblem}
\begin{ditorproblem}\label{ditprob2}
Is there a join-semilattice of breadth $2$ and cardinality $\aleph_{\omega+1}$ whose principal ideals have cardinality $< \aleph_\omega$?
\end{ditorproblem}}

 Recently, progress was made in addressing \fxnote*{}{Ditor's Problems and the remaining cases of Ditor's questions \ref{questionA} and \ref{questionB}}. Wehrung \cite{MR2609217} provided a positive answer in $\mathsf{ZFC}$ to question~\ref{questionA} when $\kappa$ is uncountable regular. Moreover, under some additional set-theoretic assumptions, he constructed a $3$-ladder of cardinality $\aleph_2$---thus, giving a consistent positive answer to \fxnote*{}{Problem~\ref{ditprob1} or, in other words, to question~\ref{questionB} when $n = 2$}. These set-theoretic assumptions, both independent of $\mathsf{ZFC}$, consist of either $\mathsf{MA}(\aleph_1; \text{precaliber } \aleph_1)$---i.e. a weakening of Martin's axiom $\mathsf{MA}(\aleph_1)$---or the existence of an $(\omega_1, 1)$-morass. Notably, since the non-existence of $(\omega_1,1)$-morasses implies that $\omega_2$ is inaccessible in $\mathbf{L}$ \cite{MR750828}, it follows from Wehrung's result that if there are no $3$-ladders of size $\aleph_2$, then $\omega_2$ is inaccessible in $\mathbf{L}$. Let us emphasize that we do not yet know whether the existence of a $3$-ladder of cardinality $\aleph_2$ is a theorem of $\mathsf{ZFC}$ (see Section~\ref{sec:conclusions} for more open questions).

We introduce a generalization of the notion of $n$-ladder that encompasses join-semilattices that are not lower finite, as to deal with both of Ditor's questions simultaneously. In particular, we introduce the concepts of $(n, \kappa)$-semiladder and $(n, \kappa)$-ladder, where $\kappa$ is an infinite cardinal and $n\in\omega$: an $(n, \kappa)$-semiladder (resp. $(n, \kappa)$-ladder) is a join-semilattice (resp. lattice) whose principal ideals have cardinality $< \kappa$, and that satisfies a property similar to ``every element has at most $n$ lower covers'' (see Definition~\ref{def:generalladder}). The notion of $(n, \aleph_0)$-ladder coincides with the standard notion of $n$-ladder, and each $(n, \kappa)$-semiladder has breadth at most $n$ (Lemma~\ref{lemma:breadth}).

Here are our main results: 

\begin{theorem}\label{thm:main1}
For every uncountable regular cardinal $\kappa$ and $n\in\omega$, there exists an $(n+1, \kappa)$-semiladder of cardinality $\kappa^{+n}$.
\end{theorem}

\begin{theorem}\label{thm:main2}
Let $\kappa$ be a singular cardinal and $n \in \omega$. If $\square_{\kappa^{+m}}(\sqleft{\mathrm{cf}(\kappa)})$ holds for every $m < n$, then there exists an $(n+1, \kappa)$-semiladder of cardinality $\kappa^{+n}$.
\end{theorem}

\begin{theorem}\label{thm:main3}
Let $\kappa$ be an infinite cardinal and $n \in \omega$. If $\square_{\kappa^{+m}}$ holds for every $m < n$, then there exists an $(n+1, \kappa)$-ladder of cardinality $\kappa^{+n}$.
\end{theorem}

\fxnote*{}{The (join-semi)lattices we construct to witness the conclusions of Theorems~\ref{thm:main1}--\ref{thm:main3} are well-founded; moreover, they can also be taken to have a least element, by adjoining one if necessary.}

Theorem~\ref{thm:main1} yields that question~\ref{questionA} has a positive answer in $\mathsf{ZFC}$ when $\kappa$ is uncountable regular, as already shown by Wehrung~\cite{MR2609217}, although we prove it by constructing more regular witnesses (i.e. semiladders). 

Theorem~\ref{thm:main2} implies that, consistently, question~\ref{questionA} has a positive answer also when $\kappa$ is a singular---the combinatorial principle $\square_{\lambda}(\sqleft{\chi})$, which is defined at the beginning of Section~\ref{sec:thm2}, is a weakening of Jensen's $\square_{\lambda}$, and as such holds under the axiom of constructibility ($\mathbf{V=L}$). \fxnote*{}{In particular, it entails that Ditor's Problem~\ref{ditprob2} has, consistently, a positive solution:

\begin{corollary}\label{cor:two}
If $\square_{\aleph_\omega}$ holds, then there exists a join-semilattice of breadth $2$, cardinality $\aleph_{\omega+1}$, and whose principal ideals have cardinality $< \aleph_\omega$.
\end{corollary}

The failure of $\square_{\aleph_\omega}$ implies the existence of an inner model with a Woodin cardinal\footnote{If, moreover, $\aleph_\omega$ is strong limit, then the failure of $\square_{\aleph_\omega}$ implies $\mathsf{AD}^{\mathbf{L}(\mathbb{R})}$ \cite{MR2194247} and more \cite{MR3225589}.} \cite{MR1359965}. Thus, it follows directly from Corollary~\ref{cor:two} that a negative solution to Ditor's Problem~\ref{ditprob2} has great consistency strength.}

Theorem~\ref{thm:main3} has the most demanding hypotheses, but also yields the strongest result: question~\ref{questionB} has \fxnote*{}{a} positive answer for every $n > 0$, and \fxnote*{}{question}~\ref{questionA} has \fxnote*{}{a} positive answer for every $\kappa$ and $n > 0$, with the witnesses being not only join-semilattices, but lattices. \fxnote*{}{In particular, Theorem~\ref{thm:main3} entails that a positive solution to Ditor's Problem~\ref{ditprob1} follows from $\square_{\aleph_1}$---recall that Wehrung also gave a consistent positive solution to this problem using an $(\omega_1,1)$-morass \cite{MR2609217}.
\begin{corollary}\label{cor:first}
If $\square_{\aleph_1}$ holds, then there exists a $3$-ladder of cardinality $\aleph_2$.
\end{corollary}}
Furthermore, since it is well-known that the failure of $\square_{\aleph_1}$ implies that $\omega_2$ is Mahlo in $\mathbf{L}$ \cite{MR750828}, we obtain the following strengthening of Wehrung's consistency result as direct consequence of Corollary~\ref{cor:first}.

\begin{corollary}
If there is no $3$-ladder of cardinality $\aleph_2$, then $\omega_2$ is Mahlo in $\mathbf{L}$.
\end{corollary}

\fxnote*{}{More generally, Theorem~\ref{thm:main3} implies that Ditor's questions~\ref{questionA} and \ref{questionB} have positive answers under the axiom of constructibility:}
\begin{corollary}
Assume $\mathbf{V = L}$. Then:
\begin{enumerate}[label={\upshape (\alph*)}]
\itemsep0.3em
\item For every $n\in\omega$ and infinite cardinal $\kappa$, there exists a join-semilattice of breadth $n+1$ and cardinality $\kappa^{+n}$ whose principal ideals have cardinality $< \kappa$.
\item For every $n \in \omega$, there exists an $(n+1)$-ladder of cardinality $\aleph_n$.
\end{enumerate}
\end{corollary}

\fxnote*{}{Lastly, note that Theorem~\ref{thm:main3} recovers Ditor's result when $\kappa = \aleph_0$ and $n = 1$: there exists a $2$-ladder of cardinality $\aleph_1$---indeed, $\square_\omega$ trivially holds.} 

In Section~\ref{sec:preliminaries}, we discuss some preliminary results. In Section~\ref{sec:generalladder} we define and examine $(n, \kappa)$-ladders.
In Section~\ref{sec:thm1}, we prove Theorem~\ref{thm:main1}. In Section~\ref{sec:specialladder} we introduce the notion of \emph{special} $(n, \kappa)$-ladder, which is crucial for employing the square machinery to prove Theorems~\ref{thm:main2} and \ref{thm:main3}, whose proofs are contained in  Sections~\ref{sec:thm2} and \ref{sec:thm3}, respectively. Finally, in the last section, we address some open questions.

\section{Preliminaries}\label{sec:preliminaries}

\subsection{Notation and basic concepts}
The monographs \cite{MR1940513} and \cite{MR2768581} are our references for all classical definitions and notation in set theory and lattice theory, respectively. 

\fxnote*{}{Given an ordinal $\alpha$, $\mathrm{cf}(\alpha)$ denotes the \emph{cofinality} of $\alpha$, that is the least order type of a cofinal subset of $\alpha$ (equivalently, the least cardinality of an unbounded subset of $\alpha$). For a well-ordered set $X$, $\mathrm{otp}(X)$ denotes its order type, i.e. the unique ordinal to which $X$ is isomorphic. For a given set $X$, and a cardinal $\lambda$, we denote by $[X]^{\lambda}$ and $[X]^{<\lambda}$ the sets of all the subsets of $X$ of cardinality exactly $\lambda$ and less than $\lambda$, respectively.}

\fxnote*{}{A \emph{join-semilattice} is a nonempty set equipped with a binary, associative, commutative, and idempotent operation, then denoted by $\vee$; it induces a partial order via $x \le y \iff x \vee y = y$. Equivalently, a join-semilattice is a partially ordered set in which every pair of elements $x,y$ admits a least upper bound, denoted by $x \vee y$. The dual notion is the \emph{meet-semilattice}.}  We treat semilattices as algebraic structures or as posets depending on what representation is more suited for the given context.

Given a join-semilattice $(P, \le)$ and some $p \in P$, we denote by $P \downarrow p$ and $P \uparrow p$ the sets $\{q \in P \mid q \le p\}$ and $\{q \in P \mid q \ge p\}$, respectively. Sometimes, instead of $P \downarrow p$ we write ${\le} \downarrow p$ or simply ${\downarrow} p$, when no ambiguity arises. A nonempty subset $D \subseteq P$ is \emph{downward closed} if for every $p \in D$, ${\downarrow} p \subseteq D$. Recall also that a nonempty subset $S\subseteq P$ is called a \emph{join-subsemilattice} of $P$ if it is closed under \fxnote*{}{binary} joins. Furthermore, a join-subsemilattice which is also downward closed is called an \emph{ideal}. Note that $P \downarrow p$ is an ideal of $P$ for every $p \in P$; such ideals are known as \emph{principal ideals}. An ideal $I$ that does not coincide with the whole join-semilattice is called a \emph{proper ideal}.

The greatest element (resp. least element) of a join-semilattice $P$, if it exists, is denoted by $\mathbf{1}_P$ (resp. $\mathbf{0}_P$), or simply $\mathbf{1}$ (resp. $\mathbf{0}$) if there is no risk of ambiguity. If we write $\bigvee \emptyset$, we are tacitly assuming that $P$ has a least element, and therefore that $\bigvee \emptyset = \mathbf{0}$. Furthermore, given two elements $p,q \in P$, $q$ is a \emph{lower cover of $p$} if $q < p$ and there is no $x \in P$ with $q < x < p$. 

Given two posets $(P, \le)$ and $(Q, \le)$, a map $f: P \rightarrow Q$ is \emph{isotone}  if $p \le p'$ implies $f(p) \le f(p')$ for every $p,p' \in P$. A map $f: P\rightarrow Q$ is said to be an \emph{order-embedding} if it is isotone, injective, and its inverse is isotone. An \emph{isomorphism} is a surjective order-embedding. \fxnote*{}{For readability, when working with linear orders we use the terms \emph{increasing} and \emph{strictly increasing} to mean isotone and order-embedding, respectively.} Furthermore, when $P$ and $Q$ are join-semilattices, $f$ is a \emph{join-homomorphism} if $f(p \vee p') = f(p) \vee f(p')$ for every $p,p'\in P$. An injective join-homomorphism is called a \emph{join-embedding}. Note that every join-homomorphism is isotone and that every join-embedding is an order-embedding, but the converse is not true in general. \fxnote*{}{We write $f: X \hookrightarrow Y$ to mean that $f$ is an embedding from $X$ into $Y$ (either a join-embedding or an order-embedding, as specified each time).}

Let us briefly review the notion of quotient join-semilattice. Given a join-semilattice $( P, \vee)$, an equivalence relation $\sim$ on $P$ is a \emph{congruence relation} if for all $x_0,x_1,y_0,y_1$ in $P$,
\[
x_0 \sim y_0 \text{ and } x_1 \sim y_1 \Rightarrow x_0 \vee x_1 \sim y_0 \vee y_1.
\]

Given a congruence relation $\sim$ on $P$, we can define the join operator $\vee$ on the quotient $P/{\sim}$ as follows: for every $x,y \in P$,
\[
[x]_\sim \vee [y]_\sim \coloneqq [x\vee y]_\sim
\]
It is easy to check that this operator satisfies all the properties of a join. The resulting join-semilattice $P / {\sim}$ is called the \emph{quotient join-semilattice of $P$ modulo~$\sim$}. We denote the quotient map by $\pi_\sim : P \rightarrow P/{\sim}$.

Any ideal $I$ of $P$ induces the following natural congruence relation $\sim_I$ on $P$: for $x,y \in P$, $x \sim_I y$ if there exists $z \in I$ such that $x \vee z = y \vee z$. In this case, we simply write $P/I$ and $\pi_I$ instead of $P/{\sim_I}$ and $\pi_{\sim_I}$. Not every congruence relation on a join-semilattice is induced by an ideal.

\begin{definition}
Let $P$ be a join-semilattice and $n\in\omega$. We say that $P$ has \emph{breadth at most $n$} if, for every nonempty finite subset $X$ of  $P$, there exists $Y\subseteq X$ with at most $n$ elements such that $\bigvee X = \bigvee Y$. The \emph{breadth} of $P$ is the least $n\in\omega$ such that $P$ has breadth at most $n$, if such $n$ exists.
\end{definition}

\fxnote*{}{ In fact, there is a more general notion of breadth which is self-dual and purely poset-theoretical \cite[\S 4]{MR0732199}, but the definition given above is more convenient for our purposes.}
Furthermore, note that a join-semilattice $P$ has breadth $0$ if and only if  $P = \{\mathbf{0}\}$, and has breadth at most $1$ if and only if it is a linear order. \fxnote*{}{The next lemma is immediate from the definition.}

\begin{lemma}
Given a join-semilattice $P$ and an $n\in\omega$, the following are equivalent:
\begin{enumerate}[label={\upshape (\arabic*)}]
\itemsep0.3em
\item $P$ has breadth at most $n$.
\item For every $X \in [P]^{n+1}$, there exists $Y \in [X]^n$ such that $\bigvee X = \bigvee Y$.
\end{enumerate}
\end{lemma}

\subsection{Ditor's theorem}
We have already mentioned the following theorem at the beginning of the introduction, as it underpins questions~\ref{questionA} and \ref{questionB}. It shows that the breadth, together with the cardinality of the principal ideals, provides a neat upper bound on the cardinality of the join-semilattice.

\begin{theorem}[{Ditor, \cite{MR0732199}}]\label{thm:Ditor}
Given some $n\in\omega$ and an infinite cardinal $\kappa$, if $P$ is a join-semilattice of breadth at most $n + 1$ whose principal ideals have cardinality $< \kappa$, then
\begin{enumerate}[label={\upshape (\alph*)}]
\itemsep0.3em
\item\label{thm:Ditor-1}
\( |P| \le \kappa^{+n} \), and
\item \label{thm:Ditor-2}
\( |I| < \kappa^{+n} \) for every proper ideal $I$ of $P$.
\end{enumerate}
\end{theorem}

As an application of Theorem~\ref{thm:Ditor}, let us prove the following proposition, which states that every join-semilattice witnessing the sharpness of Theorem~\ref{thm:Ditor}\ref{thm:Ditor-1} for some $n$ and some cardinal $\kappa$, must be $\mathrm{cf}(\kappa)$-directed. Recall that a poset $P$ is said to be \emph{$\mu$-directed}, for some infinite cardinal $\mu$, if every subset of $P$ of cardinality $< \mu$ has an upper bound.

\begin{proposition}\label{prop:directed}
Given some $n\in\omega$ and an infinite cardinal $\kappa$, every join-semilattice of cardinality $\kappa^{+n}$, breadth \fxnote*{}{at most} $n+1$, and whose principal ideals have cardinality $ <\kappa$ is $\mathrm{cf}(\kappa)$-directed.
\end{proposition}
\begin{proof}

Let $P$ be a join-semilattice of cardinality $\kappa^{+n}$, breadth at most $n+1$, and whose principal ideals have cardinality  $< \kappa$. Fix a subset $A\subseteq P$ of cardinality $< \mathrm{cf}(\kappa)$, towards showing that $A$ has an upper bound.

If $n = 0$, the claim is easy. Indeed, in this case, $P$ is a linear order, and, since every principal ideal has cardinality $ < \kappa$, $A$ must be bounded from above in $P$.

So assume $n > 0$. Let $J$ be the ideal generated by $A$, i.e. $$J \coloneqq \bigcup \Big\{{\big\downarrow} \bigvee F \bigm| F \in [A]^{<\omega}\Big\}.$$ 

Let us focus on the quotient $P/J$. Let us first observe that every principal ideal of $P/J$ has cardinality  $< \kappa$. Indeed, for every $p\in P$,
\begin{equation}\label{eq:directed}
{\downarrow} [p]_J = \pi_J \Big[\bigcup \big\{{\big\downarrow} \bigvee \big(F \cup \{p\}\big) \bigm| F \in [A]^{<\omega}\big\}\Big]
\end{equation}
and the cardinality of the set on the right-hand side of \eqref{eq:directed} is  $< \kappa$, being the union of  $< \mathrm{cf}(\kappa)$-many sets of cardinality $< \kappa$. \textit{A fortiori}, each class $[p]_J$ has cardinality $< \kappa$, and therefore, by the regularity of $\kappa^{+n}$, $|P/J| = |P| = \kappa^{+n}$.

By Ditor's Theorem~\ref{thm:Ditor}\ref{thm:Ditor-1}, the join-semilattice $P / J$ \fxnote*{}{has breadth greater than $n$}. Hence, we can fix a set $X \in [P / J]^{n+1}$ such that for all $Y \in [X]^n$, $\bigvee Y \neq \bigvee X$. Fix $K \in [P]^{n+1}$ such that $\pi_J[K] = X$. We now show that $\bigvee K$ is an upper bound of $A$.

Fix any $p \in A$. It follows from our assumption on $X$ and from the definition of the quotient join that, for every $L \in [K]^{n}$,
\begin{equation*}
\begin{split}
\pi_J\left(p \vee \bigvee K\right) = \pi_J(p) \vee \bigvee \pi_J[K] &= \bigvee X\\ &\neq \bigvee \pi_J[L]\\ 
&= \pi_J\left(\bigvee L\right) \vee \pi_J(p) = \pi_J\left(p \vee \bigvee L\right)
\end{split}
\end{equation*}
In particular, $p \vee \bigvee K \neq p \vee \bigvee L$ for every $L \in [K]^n$.
Consequently, $p\not\in K$ or, equivalently, $K \cup \{p\}$ has size $n+2$. Since $P$ has breadth at most $n+1$ by hypothesis, at least one of the following must hold:
\begin{enumerate}[label={(\alph*)}]
\itemsep0.3em
\item\label{item:directed} $\bigvee K = p \vee \bigvee K$.
\item There is an $L \in [K]^n$ such that $p \vee \bigvee L = p \vee \bigvee K$
\end{enumerate}
Only \ref{item:directed} is possible by our previous considerations. Thus, we have shown that for every $p \in A$, $\bigvee K = p \vee \bigvee K$ or, equivalently, that $\bigvee K$ is an upper bound of $A$.
\end{proof}

\section{Generalizing ladders}\label{sec:generalladder}

An \emph{$n$-ladder} is a lower finite lattice whose elements have at most $n$ lower covers. For lower finite lattices, the property of having at most $n$ lower covers for each element is (strictly) stronger than the property of having breadth at most $n$ \cite[Proposition 4.1]{MR0732199}. However, this implication does not hold in general for join-semilattices that are not lower finite. For example, consider the set of rational numbers $\mathbb{Q}$ with its usual ordering: it is, in particular, a join-semilattice of breadth $1$, yet no rational number has a lower cover. 

In this section, we present a notion that generalizes the notion of $n$-ladders by encompassing (join-semi)lattices that are not lower finite, while retaining all the main features of $n$-ladders. This notion is being \emph{$\MJ{(n+1)}$-free}.

Given a set $X$, we slightly modify\footnote{The correct notation would be $\overline{X}_{\scriptscriptstyle\top}$, but it quickly becomes cumbersome when the underlying set (i.e. $X$) has a long expression.} Davey and Priestley's notation \cite{MR1902334} and denote by $\MJ{X}$ the join-semilattice whose domain is $X \sqcup \{\mathbf{1}\}$ and such that $x \vee y = \mathbf{1}$ for all distinct $x,y \in X \sqcup \{\mathbf{1}\}$. For example, given an $n\in\omega$, the Hasse diagram of $\MJ{n}$ is depicted in Figure 1.

\begin{figure}[H]
\centering
\begin{tikzpicture}
    \node [draw, shape = circle, fill = black, minimum size = 0.1cm, inner sep=0pt, label={$ \mathbf{1}$}] at (0,1) (t){};
  
    \node [draw, shape = circle, fill = black, minimum size = 0.1cm, inner sep=0pt, label={[yshift=-0.5cm]$\scriptstyle 0$}] at (-2,0) (a2){};
    \node [draw, shape = circle, fill = black, minimum size = 0.1cm, inner sep=0pt, label={[yshift=-0.5cm]$\scriptstyle 1$}] at (-1.2,0) (a3){};
    \node [draw, shape = circle, fill = black, minimum size = 0.1cm, inner sep=0pt, label={[yshift=-0.5cm]$\scriptstyle 2$}] at (-0.4,0) (a4){};
\node [draw, shape = circle, fill = black, minimum size = 0.1cm, inner sep=0pt, label={[yshift=-0.52cm]$\scriptstyle n-2$}] at (1,0) (a5){};
\node [] at (0.3,0) (){$\ldots$};
    \node [draw, shape = circle, fill = black, minimum size = 0.1cm, inner sep=0pt, label={[yshift=-0.5cm]$\scriptstyle n-1$}] at (2,0) (a6){};
    
    \draw[ultra thin] (a2)--(t);
    \draw[ultra thin] (a3)--(t);
    \draw[ultra thin] (a4)--(t);
    \draw[ultra thin] (a5)--(t);
    \draw[ultra thin] (a6)--(t);
\end{tikzpicture}

\caption{Hasse diagram of $\MJ{n}$}
\end{figure}
The join-semilattice $\MJ{n}$ is the unique, up to isomorphism, join-semilattice of cardinality $n+1$ and length\footnote{The length of a poset $P$ is defined as $\sup \{|C|-1 \mid C \text{ is a chain of }P\}$.} at most $1$. We call a join-semilattice \emph{$\MJ{n}$-free} if it does not have a join-subsemilattice isomorphic to $\MJ{n}$. \fxnote*{}{Observe that a join-semilattice $P$ is $\MJ{n}$-free if and only if
\begin{multline*}
(x_i \vee x_j = x \text{ for all distinct } i,j\in \{0,1, \dots, n-1\}) \\ \Rightarrow (x_i = x \text{ for some } i \in \{0,1,\dots, n-1\})
\end{multline*}
whenever $x,x_0,x_1, \dots, x_{n-1} \in P$.}

Let us also introduce the notion of \emph{lower covering}, not to be confused with that of a lower cover. Its connection to $\MJ{n}$-free join-semilattices is the content of Proposition~\ref{prop:equivfree}.

\begin{definition}\label{def:lowercovering}
Given a join-semilattice $(P, \le)$ and an element $x \in P$, a \emph{lower covering of} $x$ is a set $\mathcal{I}$ of ideals of $P$ such that $\{y \in P \mid y < x\} = \bigcup \mathcal{I}$.
\end{definition}

Every upper-bounded ideal of a lower finite join-semilattice is a principal ideal. Therefore, if $P$ is a lower finite join-semilattice, then every lower covering of some $p \in P$ is made of principal ideals. \fxnote*{}{In particular, a lower finite join-semilattice with a least element is an $n$-ladder if and only if it is $\MJ{(n+1)}$-free. 

The next proposition trivially follows from \cite[Lemma 13.1]{MR3286148}, but we include a proof anyway for completeness.}

\begin{proposition}\label{prop:equivfree}
Given a join-semilattice $(P, \le)$ and $n\in\omega$, the following are equivalent:
\begin{enumerate}[label={\upshape (\arabic*)}]
\itemsep0.3em
\item\label{prop:equivfree1} Every element of $P$ has a lower covering of size at most $n$.
\item\label{prop:equivfree2} $P$ is $\MJ{(n+1)}$-free.
\end{enumerate}
\end{proposition}

\begin{proof}
\ref{prop:equivfree1}$\Rightarrow$\ref{prop:equivfree2}: Suppose that every element of $P$ has a lower covering of size at most $n$ and assume, towards a contradiction, that $\phi: \MJ{(n+1)} \hookrightarrow P$ is a join-embedding. By assumption, we can pick a lower covering $\mathcal{I}$ of $\phi(\mathbf{1})$ of size at most $n$. By the pigeonhole principle, there must be some $I \in \mathcal{I}$ and two distinct $i,j \le n$ such that $\phi(i), \phi(j) \in I$. Since $I$ is an ideal, $\phi(i) \vee \phi(j) \in I$, but $\phi(i) \vee \phi(j) = \phi(i \vee j) = \phi(\mathbf{1})$, and hence $\phi(\mathbf{1}) \in I$, which is a contradiction, as it means that $\phi(\mathbf{1}) < \phi(\mathbf{1})$.\vspace{0.5em}

\ref{prop:equivfree2}$\Rightarrow$\ref{prop:equivfree1}: Suppose that $P$ is $\MJ{(n+1)}$-free. Pick any $x \in P$ and let $m$ be the greatest natural number such that there exists a join-embedding $\phi: \MJ{m} \hookrightarrow P$ with $\phi(\mathbf{1}) = x$. Clearly, $m \le n$. Fix a join-embedding  $\phi: \MJ{m} \hookrightarrow P$ with $\phi(\mathbf{1}) = x$.  Now, for each $k < m$ let
\[
I_k \coloneqq \{y \in P\mid y \vee \phi(k) < x\}.
\]

We claim that $\{I_k \mid k < m\}$ is a lower covering of $x$.  Let us first prove that the $I_k$s are ideals of $P$. They are all downward closed. We now show that $I_{m-1}$ is closed under joins, as the same argument applies to all $I_k$s. Towards a contradiction, suppose that there are $y_0, y_1 \in I_{m-1}$ such that $y_0 \vee  y_1 \not\in I_{m-1}$, or equivalently, such that $y_0 \vee y_1 \vee \phi(m-1) = x$. Let $\varphi: \MJ{(m+1)} \rightarrow P$ be the map defined by: $\varphi(i) = \phi(i)$ for every $i < m-1$; $\varphi(m-1) = y_0 \vee \phi(m-1)$ and $\varphi(m) = y_1 \vee \phi(m-1)$. Then $\varphi$ is a join-embedding, against the maximality of $m$.

It remains to show that $\bigcup_{k < m} I_k = \{y \in P \mid y < x\}$. Pick some $y < x$ and assume, towards a contradiction, that $y \not\in I_k$ for every $k < m$. Equivalently, $y \vee \phi(k) = x$ for every $k < m$. We can extend $\phi$ to $\varphi: \MJ{(m+1)} \rightarrow P$ by setting $\varphi(k) = \phi(k)$ for every $k < m$ and $\varphi(m) = y$. Again, $\varphi$ is a join-embedding, against the maximality of $m$. 
\end{proof}

\begin{lemma}\label{lemma:breadth}
An $\MJ{(n+1)}$-free join-semilattice has breadth at most $n$.
\end{lemma}
\begin{proof}
\fxnote*{}{Let $X = \{x_0,x_1, \dots, x_n\}$ be a subset of an $\MJ{(n+1)}$-free join-semilattice $P$. If we let $y_i = \bigvee_{j \neq i} x_j$, then clearly $y_i \vee y_j = \bigvee X$ for all distinct $i,j \le n$. By $\MJ{(n+1)}$-freeness, there exists $i \le n$ such that $y_i = \bigvee X$. Thus, $P$ has breadth at most $n$.}
\end{proof}

Suppose that an element of a join-semilattice has a finite lower covering. In that case, the following lemma tells us that, among the finite lower coverings of the given element, one exists that is least with respect to the inclusion relation. 

\begin{lemma}
Given a join-semilattice $(P, \le)$ and some $p\in P$, if $p$ has a finite lower covering, then there exists a \textup{(}unique\textup{)} lower covering $\mathcal{I}$ of $p$ such that $\mathcal{I} \subseteq \mathcal{J}$ for every finite lower covering $\mathcal{J}$ of $p$.
\end{lemma}
\begin{proof}
Fix a finite lower covering $\mathcal{I}$ of $p$ which is $\subseteq$-minimal---i.e. such that there is no finite lower covering $\mathcal{I}'$ of $p$ with $\mathcal{I}' \subsetneq \mathcal{I}$. We want to prove that $\mathcal{I} \subseteq \mathcal{J}$ for every finite lower covering $\mathcal{J}$ of $p$. To this end, fix $I \in \mathcal{I}$ and some finite \fxnote*{}{lower} covering $\mathcal{J}$ of $p$, towards showing that $I \in \mathcal{J}$. 

We first claim that there exists some $J \in \mathcal{J}$ such that $I \subseteq J$. Suppose otherwise, towards a contradiction, and for each $J \in \mathcal{J}$ pick $p_J \in I \setminus J$. Then $\bigvee \{p_J \mid J \in \mathcal{J}\}\in I$, since $\mathcal{J}$ is finite and $I$ is closed under joins. Moreover, since $\mathcal{J}$ is a lower covering of $p$, there exists $W \in \mathcal{J}$ such that $\bigvee \{p_J \mid J \in \mathcal{J}\}\in W$.  As $W$ is downward closed, we must have $ p_W \in W$, hence the contradiction.

So fix $J \in \mathcal{J}$ such that $I \subseteq J$. We now claim that $I = J$. By the same argument of the previous paragraph, an $I' \in \mathcal{I}$ exists such that $J \subseteq I'$. In particular, $I \subseteq J \subseteq I'$. If $I \subsetneq I'$, then we contradict the minimality of $\mathcal{I}$, as the family $\mathcal{I} \setminus \{I\}$ would still be a lower covering of $p$. Hence, $I = I'$, and therefore $I = J \in \mathcal{J}$.
\end{proof}

We are ready to introduce our main notion:

\begin{definition}\label{def:generalladder}
Given an $n\in\omega$ and an infinite cardinal $\kappa$, a nonempty join-semilattice (resp. lattice) is said to be an \emph{$(n, \kappa)$-semiladder} (resp. \emph{$(n, \kappa)$-ladder}) if it is $\MJ{(n+1)}$-free and its principal ideals have cardinality $< \kappa$.
\end{definition}

In particular, it directly follows from Proposition~\ref{prop:equivfree} and our remark after Definition~\ref{def:lowercovering} that the notion of $(n, \aleph_0)$-ladder coincides with the one of $n$-ladder. \fxnote*{}{Moreover, note that a well-founded semiladder need not be a ladder; for instance, $P := \omega \cup \{a,b,\mathbf{1}\}$, where $a$ and $b$ are incomparable, above every $n < \omega$, and below $\mathbf{1}$, is a well-founded $(2,\aleph_1)$-semiladder but it is not a lattice.
}

The notion of being $\MJ{n}$-free behaves as expected with respect to products: for every two join-semilattices $P$ and $Q$ that are $\MJ{n}$-free and $\MJ{m}$-free, respectively, their product $P\times Q$ (with the product ordering) is $\MJ{(n+m-1)}$-free. Therefore, the product of an $(n, \kappa)$-semiladder with an $(m, \lambda)$-semiladder is an $(n+m, \max(\kappa, \lambda))$-semiladder. However, in the following sections, we will be interested in combining an $(n, \kappa)$-semiladder with an $(m, \lambda)$-semiladder to produce an $(n+m, \min(\kappa, \lambda))$-semiladder. This is done by using the following weaker notion of product.

\begin{definition}\label{def:quasi-product}
Let $(X, \vee)$ and $(Y, \vee)$ be join-semilattices. A join-semilattice $(X \times\nobreak Y, \vee)$ is said to be a \emph{quasi-product} of $X$ and $Y$ if the following conditions hold:
\begin{enumerate}[label={\upshape (q\arabic*)}]
\itemsep0.3em
\item\label{quasi-product-1} For every $x \in X$, the map $F_x : Y \hookrightarrow X\times Y, \ y \mapsto (x, y)$ is an order-embedding.
\item\label{quasi-product-2} The canonical projection $\pi_X : X\times Y \rightarrow X$ is a join-homomorphism.
\end{enumerate}
\end{definition}

Note that if we were to substitute \ref{quasi-product-1} with the stronger requirement ``The canonical projection $\pi_Y$ is a join-homomorphism,'' then we would end up with the standard definition of product join-semilattice. 

\begin{lemma}\label{lemma:quasi-product}
For all join-semilattices $X, Y$, for every $x \in X$ and every quasi-product of $X$ and $Y$, the map $F_x$ is a join-embedding.
\end{lemma}
\begin{proof}
Fix a quasi-product $(X\times Y, \vee)$ of $X$ and $Y$. Fix also $x\in X$ and $y, z \in Y$, towards showing that $(x, y) \vee (x, z) = (x, y \vee z)$.  Let $(p,q) \in X\times Y$ be such that $(x, y) \vee (x, z) = (p, q)$. From \ref{quasi-product-2}, it follows that $x = p$. But then, it immediately follows from \ref{quasi-product-1} that $q = y \vee z$.
\end{proof}

The following proposition shows that the notion of being $\MJ{n}$-free also behaves well with respect to quasi-products. 

\begin{proposition}\label{prop:quasi-productfree}
Let $X$ and $Y$ be join-semilattices and let $n,m$ be positive integers such that $X$ and $Y$ are $\MJ{n}$-free and $\MJ{m}$-free, respectively. Then every quasi-product of $X$ and $Y$ is $\MJ{(n+m-1)}$-free.
\end{proposition}
\begin{proof}
\fxnote*{}{Denote the positive integer $n+m-1$ by $k$. By contraposition, suppose that there exist a quasi-product $(X\times Y, \vee)$ of $X$ and $Y$ and a join-embedding $\varphi: \MJ{k} \hookrightarrow X\times Y$, towards finding $a,b\in\omega$ and two join-embeddings $\varphi_X: \MJ{a}\hookrightarrow X$ and $\varphi_Y: \MJ{b}\hookrightarrow Y$ such that $a+b = k$. In particular, it would follow that either $a \ge n$ or $b \ge m$, which would imply that either $X$ is not $\MJ{n}$-free or $Y$ is not $\MJ{m}$-free.}

Let 
\[
B \coloneqq \big\{l < k\mid \pi_X \circ \varphi(l) = \pi_X \circ \varphi(\mathbf{1})\big\}.
\]
By treating $\MJ{B}$ and $\MJ{(k\setminus B)}$ as join-subsemilattices of $\MJ{k}$, we define the maps $\varphi_X : \MJ{(k\setminus B)}\rightarrow X$ and $\varphi_Y:\MJ{B} \rightarrow Y$  as $\pi_X \circ \varphi \upharpoonright \MJ{(k\setminus B)}$ and $\pi_Y \circ \varphi \upharpoonright \MJ{B}$, respectively. We claim that both maps are join-embeddings, which suffices to finish the proof (by setting $b = |B|$ and $a = k - b = |k \setminus B|$).

By definition of $B$, we have $\pi_Y \circ \varphi \upharpoonright \MJ{B} = F_{\pi_X \circ \varphi(\mathbf{1})}^{-1} \circ \varphi \upharpoonright \MJ{B}$. Since $\varphi$ is a join-embedding by hypothesis, and $F_{\pi_X \circ \varphi(\mathbf{1})}^{-1} \upharpoonright \varphi[\MJ{B}]$ is a join-embedding by Lemma~\ref{lemma:quasi-product}, we conclude that $\varphi_Y$ is a join-embedding.

Since, by \ref{quasi-product-2}, $\pi_X$ is a join-homomorphism, the map $\varphi_X$ is a join-homomorphism. Therefore, in order to show that $\varphi_X$ is a join-embedding, it suffices to prove that  $\pi_X \upharpoonright \varphi[\MJ{(k\setminus B)}]$ is injective. Suppose otherwise. Then there must exist two distinct $i,j \in k\setminus B$ such that $\pi_X \circ \varphi(i) = \pi_X \circ \varphi(j)$. Then, by \ref{quasi-product-2}, $\pi_X(\varphi(i) \vee \varphi(j)) = \pi_X \circ \varphi(i)$, but $\varphi(i) \vee \varphi(j) = \varphi(\mathbf{1})$, and thus $\pi_X \circ \varphi(i) = \pi_X \circ \varphi(\mathbf{1})$ which contradicts $i\not\in B$. 
\end{proof}

\fxnote*{}{Furthermore, quasi-products preserve well-foundedness, an easy fact that will be useful in the next section.

\begin{lemma}\label{lemma:wellfoundedness}
Every quasi-product of two well-founded join-semilattices is well-founded.
\end{lemma}
}

\section{Proof of Theorem~\ref{thm:main1}}\label{sec:thm1}

This section is devoted to the proof of Theorem~\ref{thm:main1}. The following proposition does all the work. 

\begin{proposition}\label{prop:construction}
Given an uncountable regular cardinal $\kappa$, $n\in\omega$, and a well-founded $(n, \kappa^+)$-semiladder $(P, \le)$, there exists a quasi-product of $P$ and $\kappa$ whose principal ideals have cardinality $< \kappa$. 
\end{proposition}

\begin{proof}
We define inductively a system $\langle {\trianglelefteq_p} \mid p \in P \rangle$ of join-semilattice orderings such that:
\begin{enumerate}[label=(\roman{*})]
\itemsep0.3em
\item \label{proof:prop:construction-1}$\mathrm{dom}(\trianglelefteq_p) = (P\downarrow p) \times \kappa$.
\item \label{proof:prop:construction-2}$(p,\alpha) \trianglelefteq_p (p, \beta)$ if and only if $\alpha \le \beta$.
\item \label{proof:prop:construction-3} The principal ideals of $\trianglelefteq_p$ have cardinality $< \kappa$.
\item \label{proof:prop:construction-4}If $p \le q$ then $\trianglelefteq_p = \trianglelefteq_q \upharpoonright \mathrm{dom}(\trianglelefteq_p)$ and $\mathrm{dom}(\trianglelefteq_p)$ is an ideal of $\trianglelefteq_q$.
\end{enumerate}
Suppose we have defined such a system and let ${\trianglelefteq} \coloneqq \bigcup_{p \in P} {\trianglelefteq_p}$. We denote the join operator associated to $\trianglelefteq_p$ and $\trianglelefteq$ by $\vee_p$ and $\vee$, respectively. We also denote the join operator of $P$ by $\vee$, but this causes no ambiguity in what follows.
\begin{claim}\label{claim:prop:construction}
$(P\times \kappa, \trianglelefteq)$ is a quasi-product of $P$ and $\kappa$ whose principal ideals have cardinality $<\kappa$.
\end{claim}
\begin{proof}

By condition~\ref{proof:prop:construction-4}, $(P\times \kappa, \trianglelefteq)$ is a direct limit of join-semilattices. In particular, $(P\times \kappa, \trianglelefteq)$ is a join-semilattice. Furthermore, as a direct consequence of \ref{proof:prop:construction-3} \fxnote*{}{and \ref{proof:prop:construction-4}}, its principal ideals have cardinality less than $\kappa$.

Let us show that $(P\times \kappa, \trianglelefteq)$ is a quasi-product of $P$ and $\kappa$. Condition~\ref{quasi-product-1} directly follows from \ref{proof:prop:construction-2}. In order to show that \ref{quasi-product-2} holds, fix $(p_0, \alpha_0)$ and $(p_1, \alpha_1)$ in $P\times \kappa$, towards proving that there exists $\beta \in \kappa$ with $(p_0 \vee p_1, \beta) = (p_0, \alpha_0) \vee (p_1, \alpha_1)$. Let $q \in P$ and $\beta \in \kappa$ be such that $(p_0, \alpha_0) \vee (p_1, \alpha_1) = (q, \beta)$. By \ref{proof:prop:construction-1},  $(p_0, \alpha_0)$ and $(p_1, \alpha_1)$ belong to the domain of $\trianglelefteq_{p_0 \vee p_1}$. By \ref{proof:prop:construction-4}, their join (in $(P\times \kappa, \trianglelefteq)$) \fxnote*{}{also} belongs to the domain of $\trianglelefteq_{p_0 \vee p_1}$. In particular, $q \le p_0 \vee p_1$. On the other hand, by \ref{proof:prop:construction-4} again, $(p_0, \alpha_0)$ and $(p_1, \alpha_1)$ must belong to the domain of $\trianglelefteq_q$, and thus, by \ref{proof:prop:construction-1}, $p_0, p_1 \le q$, or equivalently $p_0 \vee p_1 \le q$. Therefore $q = p_0 \vee p_1$ as we wanted to show.
\end{proof}

We proceed with the inductive definition \fxnote*{}{of the $\trianglelefteq_p$}. We carry out the induction over $(P, \le)$, which is well-founded by hypothesis. First let $(p, \alpha) \trianglelefteq_{p} (p, \beta)$ for every minimal $p\in P$ and every $\alpha \le \beta$. 

Now fix $p \in P$ and suppose that we have defined $\trianglelefteq_q$ for every $q < p$, towards defining $\trianglelefteq_p$. By hypothesis and Proposition~\ref{prop:equivfree}, $p$ has a lower covering of size at most $n$. Fix a lower covering $\{I_0, I_1, \ldots, I_{n-1}\}$ of $p$ \fxnote*{}{(with repetitions if necessary)}.

\begin{claim}\label{claim:construction}
There exists a sequence $\langle J_\alpha^i \mid i < n \text{ and } \alpha < \kappa\rangle$ such that, for every $q,z \in P$ and $\alpha,\beta,\gamma <\kappa$ and $i < n$:
\begin{enumerate}[label={\upshape (\alph*)}]
\itemsep0.3em
\item\label{claim:construction-1} $\bigcup_{\alpha \in \kappa} J^i_\alpha = I_i\times \kappa$.
\item\label{claim:construction-2} If $\alpha \le \beta $, then $J^i_\alpha \subseteq J^i_\beta$.
\item\label{claim:construction-3} $|J^i_\alpha| < \kappa$.
\item\label{claim:construction-4} If $(q, \beta) \trianglelefteq_{z} (z, \gamma) \in J_\alpha^i$, then $(q, \beta) \in J_\alpha^i$.
\item\label{claim:construction-5} If $q \vee z \in I_i$ and $(q, \beta), (z, \gamma) \in J_\alpha^0 \cup \ldots \cup J_\alpha^{n-1}$, then $(q, \beta) \vee_{q \vee z} (z, \gamma) \in J_\alpha^i$.
\end{enumerate}
\end{claim}
\begin{proof}
We construct this sequence by induction on $\alpha < \kappa$. First, for each $i < n$ fix an enumeration $\langle (q_\alpha^i, \gamma_\alpha^i) \mid \alpha < \kappa \rangle$ of $I_i\times \kappa$. Set $J_0^i = {\trianglelefteq_{q_0^i}} \downarrow (q_0^i, \gamma_0^i)$. Suppose we have defined $J^i_\beta$ for every $i < n$ and $\beta < \alpha$. Define the sequence $\langle H_k^i \mid i < n \text{ and } k \in \omega \rangle$ as follows: for each $i < n$, let 
\[
H_0^i = \big({\trianglelefteq_{q_\alpha^i}} \downarrow (q_\alpha^i, \gamma_\alpha^i)\big)\cup \bigcup_{\beta < \alpha} J_\beta^i.
\]
Then, for each \fxnote*{}{$i < n$ and $k \in \omega$}, we let inductively
\begin{multline*}
H_{k+1}^i = H_{k}^i \cup \bigcup\Big\{{\trianglelefteq_{z_0\vee z_1}}\downarrow \big((z_0, \beta_0) \vee_{z_0 \vee z_1} (z_1, \beta_1)\big) \Bigm| \\(z_0, \beta_0), (z_1, \beta_1) \in \bigcup_{l < n} H^l_k \text{ and } z_0 \vee z_1 \in I_i\Big\}.
\end{multline*}
Finally let $J_\alpha^i = \bigcup_{k \in \omega} H_k^i$. The sequence $\langle J_\alpha^i \mid i < n \text{ and } \alpha \in \kappa \rangle$ satisfies \ref{claim:construction-1}-\ref{claim:construction-5}: properties~\ref{claim:construction-1} and \ref{claim:construction-2} directly follow from our construction \fxnote*{}{of $H_0^i$ already}; properties \ref{claim:construction-4} and \ref{claim:construction-5} are guaranteed by the way we defined the $H_k^i$s; finally, to see that also property~\ref{claim:construction-3} holds, note that each $H_k^i$ has cardinality less than $\kappa$ by the regularity of $\kappa$, and that $J_\alpha^i$, being the union of the $H_k^i$s, has cardinality less than $\kappa$ since $\kappa$ is uncountable and regular.
\end{proof}

We are ready to define $\trianglelefteq_p$:
\begin{itemize}
\itemsep0.3em
\item let $(p,\alpha) \trianglelefteq_p (p,\beta)$ for every $\alpha\le\beta<\kappa$, 
\item let $(q, \beta) \trianglelefteq_p (p,\alpha)$ for every $\alpha < \kappa$ and $i < n$ and $(q, \beta) \in J_\alpha^i$,
\item \fxnote*{}{let $(q, \beta) \trianglelefteq_p (z, \gamma)$ if there exists $r < p$ such that  $(q, \beta) \trianglelefteq_r (z, \gamma)$.}
\end{itemize}

First let us verify that our definition is well-behaved.
\fxnote*{}{
\begin{claim}\label{claim:behave}
$\trianglelefteq_p \upharpoonright \mathrm{dom}(\trianglelefteq_q) = {\trianglelefteq_q}$ for every $q < p$.
\end{claim}
\begin{proof}
The inclusion $\trianglelefteq_p \upharpoonright \mathrm{dom}(\trianglelefteq_q) \supseteq {\trianglelefteq_q}$ trivially follows from the definition of $\trianglelefteq_p$. Hence we are left to prove the other inclusion, namely $\trianglelefteq_p \upharpoonright \mathrm{dom}(\trianglelefteq_q) \subseteq {\trianglelefteq_q}$.

Fix some $q < p$ and $(z_0, \beta_0), (z_1, \beta_1) \in \mathrm{dom}(\trianglelefteq_q)$ such that $(z_0, \beta_0) \trianglelefteq_p (z_1, \beta_1)$. Clearly, since $(z_0, \beta_0), (z_1, \beta_1) \in \mathrm{dom}(\trianglelefteq_q)$, we have $z_0 \vee z_1 \le q$. Moreover, by definition of $\trianglelefteq_p$, there is some $r < p$ such that $(z_0, \beta_0) \trianglelefteq_r (z_1, \beta_1)$. In particular, $z_0 \vee z_1 \le r$, since $(z_0, \beta_0), (z_1, \beta_1) \in \mathrm{dom}(\trianglelefteq_r)$. It follows from the induction hypothesis on $\trianglelefteq_r$ that $(z_0, \beta_0) \trianglelefteq_{z_0 \vee z_1} (z_1, \beta_1)$. Finally, as $z_0 \vee z_1 \le q$, we obtain, by induction hypothesis on $\trianglelefteq_q$, $(z_0, \beta_0) \trianglelefteq_q (z_1, \beta_1)$.
\end{proof}}

\fxnote*{}{It follows from Claim~\ref{claim:behave} together with} the transitivity of $\trianglelefteq_q$ for $q < p$ and properties~\ref{claim:construction-2} and \ref{claim:construction-4} of the $J_\alpha^i$s that $\trianglelefteq_p$ is transitive. Since $\trianglelefteq_p$ is reflexive, we conclude it is a partial order. Moreover, for every $\alpha \in \kappa$, we have, by construction, 
\[
{\trianglelefteq_p} \downarrow (p, \alpha) = \big(\{p\}\times(\alpha+1)\big) \cup \bigcup_{i < n} J_\alpha^i
\]
which has cardinality less than $\kappa$ by \ref{claim:construction-3}. In particular, every principal ideal of $\trianglelefteq_p$ has cardinality less than $\kappa$.
 
Towards showing that $\trianglelefteq_p$ is a join-semilattice, fix $(q_0,\beta_0), (q_1, \beta_1)$ with $q_0,q_1 \le p$. Suppose first that $q_0 \vee q_1 < p$. We claim \fxnote*{}{that} $(q_0, \beta_0) \vee_{q_0 \vee q_1} (q_1, \beta_1)$ is the $\trianglelefteq_p$-least upper bound of $\{(q_0, \beta_0), (q_1, \beta_1)\}$---note that the claim also implies that $\mathrm{dom}(\trianglelefteq_q)$ is an ideal of $\trianglelefteq_p$ for every $q < p$. It suffices to show that if $(q_0, \beta_0), (q_1, \beta_1) \trianglelefteq_p (p, \alpha)$ for some $\alpha$, then $(q_0, \beta_0) \vee_{q_0 \vee q_1} (q_1, \beta_1) \trianglelefteq_p (p, \alpha)$\fxnote*{}{---indeed, the other cases are made trivial by Claim~\ref{claim:behave}}. By hypothesis and by definition of $\trianglelefteq_p$, there must be $i, j, k < n$ such that $q_0 \vee q_1 \in I_i$, $(q_0, \beta_0)\in J_\alpha^j$ and $(q_1, \beta_1) \in J_\alpha^k$. But then, it follows from property~\ref{claim:construction-5} of the $J_\alpha^i$s that $(q_0, \beta_0) \vee_{q_0 \vee q_1} (q_1, \beta_1) \in J_\alpha^i$. Thus, $(q_0, \beta_0) \vee_{q_0 \vee q_1} (q_1, \beta_1) \trianglelefteq_p (p, \alpha)$ by definition of $\trianglelefteq_p$.

If, on the other hand, $q_0 \vee q_1 = p$, then $$\Big(p, \min\big\{\alpha \in \kappa \mid (q_0, \beta_0),(q_1, \beta_1) \trianglelefteq_p (p,\alpha)\big\}\Big)$$ is easily seen to be the $\trianglelefteq_p$-least upper bound of $\{(q_0, \beta_0), (q_1, \beta_1)\}$---note that, by property~\ref{claim:construction-1} of the $J_\alpha^i$s, there always exist $\alpha \in \kappa$ such that $(q_0, \beta_0), (q_1, \beta_1) \trianglelefteq_p (p, \alpha)$.
\end{proof}

The proof of Theorem~\ref{thm:main1} follows by a repeated application of Proposition~\ref{prop:construction}.
\begin{proof}[Proof of Theorem~{\upshape\ref{thm:main1}}]
Let $\kappa$ be an uncountable regular cardinal and $n\in\omega$ some natural number. We inductively define a finite sequence $(P_i)_{i \le n}$ of join-semilattices such that $P_i$ is a well-founded $(i+1, \kappa^{+(n-i)})$-semiladder of cardinality $\kappa^{+n}$ for every $i \le n$. The definition goes as follows:
\begin{itemize}
\itemsep0.3em
\item Let $P_0 = \kappa^{+n}$ with its usual ordering---in particular, $P_0$ is a well-founded $(1, \kappa^{+n})$-semiladder.
\item For every $i < n$, let $P_{i+1}$ be a quasi-product of $P_i$ and $\kappa^{+(n-i-1)}$ whose principal ideals have cardinality less than $\kappa^{+(n-i-1)}$, which exists by Proposition~\ref{prop:construction}. Since a quasi-product of two well-founded join-semilattices is still well-founded, and since $P_i$ is well-founded, $P_{i+1}$ is also well-founded. Moreover, as $P_i$ is $\MJ{(i+2)}$-free, it follows from Proposition~\ref{prop:quasi-productfree} that $P_{i+1}$, being a quasi-product of $P_i$ and a linear order, is $\MJ{(i+3)}$-free. Thus, $P_{i+1}$ is a well-founded $(i+2, \kappa^{+(n-i-1)})$-semiladder of cardinality $\kappa^{+n}$.
\end{itemize}
At the end, $P_n$ is an $(n+1, \kappa)$-semiladder of cardinality $\kappa^{+n}$.
\end{proof}

\section{Special ladders}\label{sec:specialladder}

In Section~\ref{sec:thm1}, we have proven Theorem~\ref{thm:main1} by iterating the construction of an ${(n+1, \kappa)}$-semiladder as a quasi-product of $P$ and $\kappa$, where $P$ is an $(n, \kappa^+)$-semilad\-der. To prove Theorems~\ref{thm:main2} and \ref{thm:main3} we follow a similar strategy: we iterate the construction of an $(n+1, \kappa)$-semiladder as a quasi-product of $|P|^+$ and $P$ for some $(n, \kappa)$-semiladder $P$. These quasi-products are induced by maps from $[|P|^+]^2$ \fxnote*{}{to} $P$ that satisfy certain triangular inequalities that are reminiscent\footnote{It is unclear to us if and to what extent our Definition~\ref{def:transmap} and the subsequent results fit into Todor\v{c}evi\'{c}'s \emph{Walks on ordinals} framework \cite{MR2355670}.} of the ones satisfied by Todor\v{c}evi\'{c} $\rho$-functions \cite{MR2355670}. Consequently, we adopt his terminology in the following definition.

\begin{definition}\label{def:transmap}
Consider a join-semilattice $(B, \le)$, an ordinal $\gamma$, and a map $\varrho: [\gamma]^2\rightarrow B$, then:
\begin{itemize}
\itemsep0.3em
\item $\varrho$ is said to be \emph{transitive} if $\varrho(\alpha, \delta) \le \varrho(\alpha, \beta) \vee \varrho(\beta, \delta)$ for every $\alpha < \beta < \delta < \gamma$.
\item $\varrho$ is said to be \emph{subadditive} if $\varrho(\alpha, \beta) \le \varrho(\alpha, \delta) \vee \varrho(\beta, \delta)$ for every $\alpha < \beta < \delta < \gamma$.
\end{itemize}
\end{definition}

Given $\alpha < \beta < \gamma$ and a map $\varrho:[\gamma]^2 \rightarrow B$, we simply write $\varrho(\alpha, \beta)$ instead of $\varrho(\{\alpha, \beta\})$, and we write $\varrho \upharpoonright \alpha$ instead of $\varrho \upharpoonright [\alpha]^2$. \fxnote*{}{From now on, we assume that $B$ has a least element\footnote{This assumption is not essential, as definition \eqref{eq:defspecialladder} (where we implicitly use this assumption) and Proposition~\ref{prop:specialladder} could be reformulated without it, at the cost of a slightly more cumbersome presentation. However, the additional generality gained by dropping the assumption is negligible for our purposes.}, and we adopt the convention that $\varrho(\alpha, \alpha) = \mathbf{0}$ for every $\alpha < \gamma$.}

For each $\alpha < \gamma$ and $p \in B$ we let 
\begin{align*}
\mathrm{ht}(\alpha, p) &\coloneqq \alpha,\\
D_\varrho(\alpha, p) &\coloneqq \big\{\eta  < \alpha \mid \varrho(\eta, \alpha) \le p\big\}.
\end{align*}
\fxnote*{}{Note that the function $\mathrm{ht}$ has nothing to do with the standard ``height" function in order theory.} We are interested in the following binary relation $\trianglelefteq_\varrho$ on $\gamma \times B$ induced by $\varrho$: for every $p, q \in B$ and $\alpha, \beta < \gamma$,\vspace{0.3em}
\begin{equation}\label{eq:defspecialladder}
 (\alpha, p) \trianglelefteq_\varrho (\beta, q) \text{ if and only if } \alpha \le \beta \text{ and } p  \vee \varrho(\alpha, \beta) \le q.\vspace{0.3em}
\end{equation}

\fxnote*{}{Note how the reflexivity of $\trianglelefteq_\varrho$ directly follows from $\varrho(\alpha, \alpha) = \mathbf{0}$ for all $\alpha$s. More importantly, observe how the} relation $\trianglelefteq_\varrho$ is locally close to the product ordering. In particular, if $\varrho:[\gamma]^2\rightarrow B$ is constant with value $\mathbf{0}$, then $\trianglelefteq_\varrho$ is exactly the product ordering of $\gamma$ and $B$. 

\fxnote*{}{
The next two results study the relationship between the properties of $\varrho$ and those of the induced relation $\trianglelefteq_\varrho$. Throughout, we fix $\varrho$ and drop the subscript $\varrho$ from $\trianglelefteq$ and $D(\alpha, p)$.

The first easy lemma identifies when the structure $( \gamma \times B, \trianglelefteq)$ has an element below all others. We deliberately avoid the expression ``least element,” which is usually reserved for partial orders; at this point it is not yet clear whether, and under which conditions, $\trianglelefteq$ is a partial order (see Proposition~\ref{prop:specialladder}).

\begin{lemma}\label{lemma:minimum}
The following are equivalent:
\begin{enumerate}[label={\upshape (\arabic*)}]
\item\label{lemma:minimum-1} There exists $x \in \gamma \times B$ such that $x \trianglelefteq y$ for all $y \in \gamma \times B$.
\item\label{lemma:minimum-2} $(0, \mathbf{0}) \trianglelefteq y$ for all $y \in \gamma \times B$.
\item\label{lemma:minimum-3} $\varrho(0, \alpha) = \mathbf{0}$ for all $\alpha < \gamma$.
\end{enumerate}
\end{lemma}
\begin{proof}
The equivalence \ref{lemma:minimum-1} $\iff$ \ref{lemma:minimum-2} directly follows from $(0, \mathbf{0})$ being a minimal element of $(\gamma \times B, \trianglelefteq)$ by definition of $\trianglelefteq$ (see \eqref{eq:defspecialladder}). The implication  \ref{lemma:minimum-3} $\Rightarrow$ \ref{lemma:minimum-2} follows straightforwardly from the definition of $\trianglelefteq$ and from $\mathbf{0}$ being the least element of $B$. Finally, in order to show \ref{lemma:minimum-2} $\Rightarrow$ \ref{lemma:minimum-3}, pick $\alpha < \gamma$. Since by assumption $(0, \mathbf{0}) \trianglelefteq (\alpha, \mathbf{0})$, it follows from the definition of $\trianglelefteq$ that $\varrho(0, \alpha) = \mathbf{0}$.
\end{proof}
}

\begin{proposition}\label{prop:specialladder}
The structure $(\gamma\times B, \trianglelefteq)$ is:
\begin{enumerate}[label={\upshape (\arabic*)}]
\itemsep0.3em
\item\label{prop:specialladder-1} a poset if and only if $\varrho$ is transitive.
\item\label{prop:specialladder-2} a join-semilattice if and only if $\varrho$ is \fxnote*{}{both} transitive and subadditive. In this case, it is a quasi-product of $\gamma$ and $B$.
\item\label{prop:specialladder-3} a lattice if and only if $\varrho$ is \fxnote*{}{both} transitive and subadditive, $B$ is a lattice, $\varrho(0, \alpha) = \mathbf{0}$, and $D(\alpha, p)$ is closed\footnote{With respect to the order topology of $\alpha$.} in $\alpha$ for every $\alpha < \gamma$ and $p \in B$.
\end{enumerate}
\end{proposition}
\begin{proof}
\ref{prop:specialladder-1}: Let us first prove the ``only if'' direction. Pick some $\alpha < \beta < \delta < \gamma$. By definition of $\trianglelefteq$, \[(\alpha, \varrho(\alpha, \beta) \vee \varrho(\beta, \delta)) \trianglelefteq (\beta, \varrho(\alpha, \beta) \vee \varrho(\beta, \delta))\] and \[(\beta, \varrho(\alpha, \beta) \vee \varrho(\beta, \delta)) \trianglelefteq (\delta, \varrho(\alpha, \beta) \vee \varrho(\beta, \delta)).\] 
Since we are assuming that $\trianglelefteq$ is transitive, it follows\[(\alpha, \varrho(\alpha, \beta) \vee \varrho(\beta, \delta)) \trianglelefteq (\delta, \varrho(\alpha, \beta) \vee \varrho(\beta, \delta)),\] and therefore, by definition of $\trianglelefteq$, $\varrho(\alpha, \delta) \le \varrho(\alpha, \beta) \vee \varrho(\beta, \delta)$. Hence, $\varrho$ is transitive. 

Now to the ``if'' direction. Pick $\alpha \le \beta \le \delta < \gamma$ and $p,q,r \in B$ such that $(\alpha, p) \trianglelefteq (\beta, q) \trianglelefteq (\delta, r)$, towards showing that $(\alpha, p) \trianglelefteq (\delta, r)$. Since $\le$ on $B$ is transitive, the only non-trivial case to check is when $\alpha < \beta < \delta$. From the definition of $\trianglelefteq$ and our assumption it follows that  $p \le q \le r$ and $\varrho(\alpha, \beta) \le q$ and $\varrho(\beta, \delta) \le r$. By transitivity of $\varrho$,  $\varrho(\alpha, \delta) \le \varrho(\alpha, \beta) \vee \varrho(\beta, \delta) \le r$, and therefore $(\alpha, p) \trianglelefteq (\delta, r)$. Hence, $\trianglelefteq$ is transitive.\vspace{0.5em}

\ref{prop:specialladder-2}: Let us first prove the ``only if'' direction. 
Suppose that $(\gamma\times B, \trianglelefteq)$ is a join-semilattice, towards showing that $\varrho$ is subadditive. First, we need the following claim, which amounts to saying that $(\gamma\times B, \trianglelefteq)$ satisfies condition \ref{quasi-product-2} of being a quasi-product (see Definition~\ref{def:quasi-product}):
\begin{claim}\label{claim:specialladder-2}
The map $\mathrm{ht}: (\gamma\times B, \trianglelefteq) \rightarrow (\gamma, \le)$ is a join-homomorphism.
\end{claim}
\begin{proof}
Fix $\alpha\le \beta < \gamma$ and $p,q \in B$. Let $\eta < \gamma$ and $z \in B$ be such that $(\alpha, p) \vee (\beta, q) = (\eta, z)$. By definition of $\trianglelefteq$, $(\beta, p \vee q \vee  \varrho(\alpha, \beta))$ is an upper bound of $\{(\alpha, p), (\beta, q)\}$. Therefore, $(\eta, z) \trianglelefteq (\beta, p \vee q \vee  \varrho(\alpha, \beta))$. In particular, $\eta \le \beta$. On the other hand, as $(\beta, q) \trianglelefteq (\eta, z)$, it follows that $\beta \le \eta$. Therefore, $\eta = \beta$. We have shown that $\mathrm{ht}$ is a join-homomorphism.
\end{proof}
Note that $(\gamma\times B, \trianglelefteq)$ satisfies also \ref{quasi-product-1} by definition of $\trianglelefteq$. Therefore, $(\gamma\times B, \trianglelefteq)$ is a quasi-product of $\gamma$ and $B$.

Going back to the ``only if'' direction, fix $\alpha, \beta, \delta$ such that $\alpha < \beta < \delta < \gamma$ and let $\bar{p} = \varrho(\alpha, \delta) \vee \varrho(\beta, \delta)$. By Claim~\ref{claim:specialladder-2}, there exists some $z \in B$ such that 
\begin{equation}\label{eq:specialladder}
(\alpha, \bar{p}) \vee (\beta, \bar{p}) = (\beta, z).
\end{equation}
Note that $\bar{p} \le z$. By $\bar{p}$'s definition, $(\alpha, \bar{p}), (\beta, \bar{p}) \trianglelefteq (\delta, \bar{p})$. It follows that $(\beta, z) \trianglelefteq (\delta, \bar{p})$. Then $z \le \bar{p}$ and we conclude that $z = \bar{p}$. This, together with \eqref{eq:specialladder}, implies that $(\alpha, \bar{p}) \trianglelefteq (\beta, \bar{p})$, and thus $\varrho(\alpha, \beta) \le \bar{p} = \varrho(\alpha, \delta) \vee \varrho(\beta, \delta)$. We have shown that $\varrho$ is subadditive.

Let us prove the ``if'' direction. The transitivity of $\varrho$ and \fxnote*{}{Item}~\ref{prop:specialladder-1} of this proposition imply that $\trianglelefteq$ is a partial order. Now pick $p,q \in B$ and $\alpha \le \beta < \gamma$, towards finding a $\trianglelefteq$-least upper bound for $\{(\alpha, p), (\beta, q)\}$. We claim that $(\beta, p \vee q \vee \varrho(\alpha, \beta))$ is the $\trianglelefteq$-least upper bound of $\{(\alpha, p), (\beta, q)\}$. 

From $\trianglelefteq$'s definition, $(\alpha, p), (\beta, q) \trianglelefteq (\beta, p \vee q \vee \varrho(\alpha, \beta))$. Now suppose that $(\alpha, p), (\beta, q) \trianglelefteq (\delta, r)$ for some $\delta < \gamma$ and $r \in B$. In particular, $\alpha, \beta \le \delta$ and $p \vee \varrho(\alpha, \delta) \le r$ and $q \vee \varrho(\beta, \delta) \le r$. From the subadditivity of $\varrho$ it follows that $\varrho(\alpha, \beta) \le r$. Altogether, we have $p \vee q \vee \varrho(\alpha, \beta) \vee \varrho(\beta, \delta)\le r$, and therefore $(\beta, p \vee q \vee \varrho(\alpha, \beta)) \trianglelefteq (\delta, r)$. Thus, our claim is true.\vspace{0.5em}

\ref{prop:specialladder-3}: Let us first prove the ``only if'' direction. Towards showing that $B$ is a meet-semilattice, fix some $q,r \in B$ and let us find a $\le$-greatest lower bound of $\{q,r\}$. Let $p \in B$ and $\alpha < \gamma$ be such that $(\alpha, p) = (0, q) \wedge (0, r)$.  By $\trianglelefteq$'s definition, we must have $p \le q,r$ and $\alpha = 0$. We claim that $p$ is the $\le$-greatest lower bound of $\{q,r\}$. Pick any $x \in B$ such that $x \le q,r$. Then $(0, x) \trianglelefteq (0, q), (0, r)$. It follows  from $(0, p)$ being the greatest lower bound of $(0,q)$ and $(0,r)$ that $(0, x) \trianglelefteq (0, p)$. In particular, $x \le p$. Thus, $p$ is the $\le$-greatest lower bound of $\{q,r\}$.

Next, fix $\alpha < \gamma$ toward showing that $\varrho(0, \alpha) = \mathbf{0}$. This immediately follows once we note the following: by definition of $\trianglelefteq$ (see \eqref{eq:defspecialladder}), $(0, \mathbf{0})$ is a minimal element of $(\gamma \times B, \trianglelefteq)$; but since we are assuming $(\gamma \times B, \trianglelefteq)$ to be lower directed, this implies that $(0, \mathbf{0})$ is actually the minimum of $(\gamma \times B, \trianglelefteq)$. By Lemma~\ref{lemma:minimum}, we are done.

Now fix $\alpha < \gamma$ and $p \in B$, towards showing that $D(\alpha, p)$ is a closed subset of $\alpha$. Fix some $\eta < \alpha$ such that $\eta = \sup (D(\alpha, p) \cap \eta)$. We want to prove that $\eta \in D(\alpha, p)$.  \fxnote*{}{Let $\mu < \gamma$ and $q \in B$ be such that $(\mu, q) = (\alpha, p) \wedge (\eta, p \vee \varrho(\eta, \alpha))$. Clearly, $\mu \le \eta$. Once we show $\mu \ge \eta$ we are done, as it follows from $(\eta, q) = (\mu, q) \trianglelefteq (\alpha, p)$ and from the definition of $\trianglelefteq$ that $\varrho(\eta, \alpha) \le p$ or, equivalently, $\eta \in D(\alpha, p)$.} 

Fix $\nu \in D(\alpha, p) \cap \eta$. By definition of $D$, $\varrho(\nu, \alpha) \le p$. Since we are assuming $(\gamma \times B, \trianglelefteq)$ to be a lattice, it follows from \fxnote*{}{Item}~\ref{prop:specialladder-2} of this proposition that $\varrho$ is subadditive. By the subadditivity of $\varrho$, $\varrho(\nu, \eta) \le \varrho(\nu, \alpha) \vee \varrho(\eta, \alpha)$. Combining these last observations, we get that $\varrho(\nu, \eta) \le p \vee \varrho(\eta, \alpha)$. 
Therefore,
\begin{equation}\label{eq:propspecialladder}
\forall \nu \in D(\alpha, p) \cap \eta, \ \ (\nu, p) \trianglelefteq (\alpha, p) \wedge (\eta, p \vee \varrho(\eta, \alpha)).
\end{equation}
Moreover, it follows from \eqref{eq:propspecialladder} that \fxnote*{}{$\mu \ge \sup (D(\alpha, p) \cap \eta)$}. But since we are assuming \fxnote*{}{$\sup (D(\alpha, p) \cap \eta) = \eta$}, we conclude $\mu \ge \eta$, as we wanted to show.

Let us prove the ``if'' direction. Fix $\beta, \delta < \gamma$ with $\beta \le \delta$ and $q, r \in B$, towards showing that there exists a $\trianglelefteq$-greatest lower bound for $\{(\beta, q), (\delta, r)\}$. If $\beta = \delta$, it is straightforward that $(\beta, q \wedge r)$ is the greatest lower bound. Hence suppose that $\beta < \delta$.

Note that the set $(D(\beta, q) \cup \{\beta\}) \cap D(\delta, r)$ is a closed nonempty subset of $\delta$: it is easy to see that it is closed, as by assumption both $D(\beta, q) \cup \{\beta\}$ and $D(\delta, r)$ are closed in $\delta$; moreover, since $\varrho(0, \beta) = \varrho(0, \delta) = \mathbf{0}$, it follows that \[0 \in (D(\beta, q) \cup \{\beta\}) \cap D(\delta, r),\]
and thus $(D(\beta, q) \cup \{\beta\}) \cap D(\delta, r) \neq\emptyset$.

Let $\alpha$ be the maximum of $(D(\beta, q) \cup \{\beta\}) \cap D(\delta, r)$, which exists since the set is nonempty, closed and bounded in $\delta$.  We claim that $(\alpha, q \wedge r)$ is the $\trianglelefteq$-greatest lower bound of $\{(\beta, q), (\delta, r)\}$. Indeed, pick  an $\eta < \gamma$ and some $p \in B$ such that $(\eta, p) \trianglelefteq (\beta, q), (\delta, r)$, towards showing that $(\eta, p) \trianglelefteq (\alpha, q \wedge r)$. By definition of $\trianglelefteq$, $p \le q \wedge r$. Moreover, as $\eta$ must belong to $(D(\beta, q) \cup \{\beta\}) \cap D(\delta, r)$, we conclude that $\eta \le \alpha$.  If $\eta = \alpha$, then $(\eta, p) = (\alpha, p) \trianglelefteq (\alpha, q \wedge r)$. So suppose $\eta < \alpha$. The following holds:
\[
\varrho(\eta, \alpha) \le \varrho(\eta, \beta) \vee \varrho(\alpha, \beta) \le q,
\]
where the first inequality comes from the subadditivity of $\varrho$ and the second one follows from both $\eta$ and $\alpha$ belonging to $D(\beta, q)$. Analogously,
\[
\varrho(\eta, \alpha) \le \varrho(\eta, \delta) \vee \varrho(\alpha, \delta) \le r.
\]
Thus, $\varrho(\eta, \alpha) \le q \wedge r$. As we already noted that $p \le q \wedge r$, we conclude $(\eta, p) \le (\alpha, q \wedge r)$.
\end{proof}
\fxnote*{}{
\begin{remark}\label{rmk:closedness}
Note that, by the definition of $\trianglelefteq$, the following holds: for every $x = (\alpha, p) \in \gamma \times B$,
\[
\{\mathrm{ht}(y) \mid y \in \gamma \times B \text{ and } y \trianglelefteq x\} = D(\alpha, p) \cup \{\alpha\}.
\]
In particular, if $(\gamma \times B, \trianglelefteq)$ is a lattice, we conclude from Proposition~\ref{prop:specialladder}\ref{prop:specialladder-3} that $\{\mathrm{ht}(y) \mid y \in \gamma \times B \text{ and } y \trianglelefteq x\}$ is a closed set of ordinals for every $x \in \gamma \times B$.
\end{remark}}

\begin{corollary}\label{cor:specialladder}
If $B$ is an $\MJ{n}$-free join-semilattice and $\varrho:[\gamma]^2 \rightarrow B$ is \fxnote*{}{both} transitive and subadditive, then $(\gamma\times B, \trianglelefteq_\varrho)$ is an $\MJ{(n+1)}$-free join-semilattice.
\end{corollary}
\begin{proof}
Immediate by \fxnote*{}{Proposition~\ref{prop:specialladder}\ref{prop:specialladder-2} together with} Proposition~\ref{prop:quasi-productfree}.
\end{proof}

Consider now the following recursive definition, where a special $(0, \kappa)$-semiladder is just the trivial join-semilattice $\{\mathbf{0}\}$.
\fxnote*{}{
\begin{definition}\label{def:specialladder}
Given a positive integer $n$ and an infinite cardinal $\kappa$, an $(n, \kappa)$-(semi)ladder $S$ is \emph{special} if $S= (\gamma\times B, \trianglelefteq_\varrho)$, where $\gamma$ is an ordinal, $B$ is a special $(n-1, \kappa)$-(semi)ladder, and $\varrho:[\gamma]^2\rightarrow B$ satisfies $\varrho(0, \alpha) = \mathbf{0}_B$ for all $\alpha < \gamma$.
\end{definition}

Four remarks are in order. The first one is that the clause $\varrho(0, \alpha) = \mathbf{0}_B$ in Definition~\ref{def:specialladder} is a genuine requirement only in the semiladder case: if $S$ is a (special) ladder---hence a lattice---then $\varrho(0, \alpha) = \mathbf{0}_B$ already holds automatically by Proposition~\ref{prop:specialladder}\ref{prop:specialladder-3}. Thus, the clause needs to be stated explicitly only for special semiladders. Secondly, note that thanks to this clause, every special semiladder has a least element (see Lemma~\ref{lemma:minimum}). Thirdly, it follows from the definition of special semiladders and Proposition~\ref{prop:specialladder}\ref{prop:specialladder-2} that special semiladders, being quasi-products of well-founded join-semilattices, are themselves well-founded (see Lemma~\ref{lemma:wellfoundedness}).} Lastly, observe that a join-semilattice $(P, \trianglelefteq)$ is a special $(1, \kappa)$-semiladder if and only if there exists an ordinal $\alpha \le \kappa$ such that $P =\alpha \times \{\mathbf{0}\}$ with $(\gamma, \mathbf{0}) \trianglelefteq (\beta, \mathbf{0})$ if and only if $\gamma \le \beta< \alpha$. For this reason, we sometimes identify the special $(1, \kappa)$-semiladders with the ordinals less than or equal to $\kappa$ in the following sections.

\section{Proof of Theorem~\ref{thm:main2}}\label{sec:thm2}

In this section, we prove Theorem~\ref{thm:main2}. Let us introduce the combinatorial principle $\square_{\lambda}(\sqleft{\chi})$, where $\lambda$ and $\chi$ are infinite cardinals and $\chi \le \lambda$:

\begin{equation*}\tag*{$\square_{\lambda}(\sqleft{\chi})$}
\begin{split}
\text{There}&\text{ exists a sequence } \langle C_\alpha \mid \alpha < \lambda^+ \text{ and } \mathrm{cf}(\alpha) \ge \chi\rangle \text{ such that:}\\
\text{(i)}& \ C_\alpha \text{ is a closed and unbounded subset of }\alpha;\\
\text{(ii)}& \ \mathrm{otp}(C_\alpha) \le \lambda;\\
\text{(iii)}& \ \text{If } \beta \text{ is a limit point of }C_\alpha \text{ and } \mathrm{cf}(\beta) \ge \chi, \text{ then } C_\beta = C_\alpha \cap \beta.
\end{split}
\end{equation*}

The principle $\square_{\lambda}(\sqleft{\chi})$ is a particular instance of a more general combinatorial principle defined by Brodsky and Rinot in \cite[Definition 1.16]{MR3914943} (see also \cite[Definition 1.7]{MR3692231}). It is a weakening of Jensen's $\square_\lambda$, and as such holds under $\mathbf{V=L}$ for all infinite cardinals $\chi \le \lambda$. Furthermore, note that for every infinite cardinal $\lambda$, the principle $\square_{\lambda}(\sqleft{\omega})$ is \fxnote*{}{equivalent to $\square_\lambda$ (see \cite[Lemma IV.5.1]{MR750828})}, and that $\square_{\lambda}(\sqleft{\lambda})$ trivially holds in $\mathsf{ZFC}$.

\fxnote*{}{Theorem~\ref{thm:main2} is a direct corollary of the next theorem. For regular $\kappa$, Theorem~\ref{thm:main1} already yields $(n+1, \kappa)$-semiladders of cardinality $\kappa^{+n}$ in $\mathsf{ZFC}$. The next theorem---stated for arbitrary infinite $\kappa$---produces \emph{special} $(n+1, \kappa)$-semiladders of cardinality $\kappa^{+n}$ under additional set-theoretic hypotheses.}

\begin{theorem}\label{thm:mainplus2}
Let $\kappa$ be an infinite cardinal and $n\in\omega$. If $\square_{\kappa^{+m}}(\sqleft{\mathrm{cf}(\kappa)})$ holds for every $m < n$, then there exists a special $(n+1, \kappa)$-semiladder of cardinality $\kappa^{+n}$.
\end{theorem}
\begin{proof}
Let us fix an infinite cardinal $\kappa$. We prove the result by induction on $n\in\omega$. The case $n=0$ is trivially true, as $\kappa$ is, in particular, a special $(1, \kappa)$-ladder of cardinality $\kappa$. So let us fix a special $(n+1, \kappa)$-semiladder $(B, \le)$ of cardinality $\kappa^{+n}$ and a $\square_{\kappa^{+n}}(\sqleft{\mathrm{cf}(\kappa)})$-sequence $\langle C_\alpha \mid \alpha < \kappa^{+(n+1)} \text{ and } \mathrm{cf}(\alpha) \ge \mathrm{cf}(\kappa)\rangle$, towards constructing a special $(n+2, \kappa)$-semiladder of cardinality $\kappa^{+(n+1)}$.

We want to define a map $\varrho: [\kappa^{+(n+1)}]^2\rightarrow B$ such that, for every $\alpha < \kappa^{+(n+1)}$ and $p \in B$:

\begin{enumerate}[label=(\alph*)]
\itemsep0.3em
\item\label{condition:general1} $\varrho$ is \fxnote*{}{both} transitive and subadditive.
\item\label{condition:general2} $|D_\varrho(\alpha, p)| \le |{\downarrow}p|+ \aleph_0$.
\item\label{condition:general3} $|D_\varrho(\alpha, p)| < \kappa$.
\item\label{condition:general4} \fxnote*{}{$\varrho(0, \alpha) = \mathbf{0}_B$.}
\end{enumerate}

Note that \ref{condition:general2} implies \ref{condition:general3} when $\kappa$ is uncountable, as we are assuming that the principal ideals of $B$ have cardinality $< \kappa$. Indeed, condition \ref{condition:general3} plays a  role only when $\kappa = \aleph_0$.

Let us first argue that if we manage to construct a map $\varrho$ satisfying \ref{condition:general1}-\ref{condition:general4}, then the induced $(\kappa^{+(n+1)}\times B, \trianglelefteq_\varrho)$ is a special $(n+2, \kappa)$-semiladder. By Corollary~\ref{cor:specialladder}, $(\kappa^{+(n+1)}\times B, \trianglelefteq_\varrho)$ is an $\MJ{(n+3)}$-free join-semilattice. It remains to argue that its principal ideals have cardinality $< \kappa$.  If $\alpha < \kappa^{+(n+1)}$ and $p \in B$, then, by definition of $\trianglelefteq_\varrho$,
\[
{\trianglelefteq_\varrho} \downarrow (\alpha, p) = (D_\varrho(\alpha, p) \cup \{\alpha\}) \times ({\downarrow} p).
\]
It follows from condition~\ref{condition:general3} and our hypotheses on $B$ that the principal ideal of $(\alpha, p)$ in $(\kappa^{+(n+1)}\times B, \trianglelefteq_\varrho)$ has cardinality $< \kappa$.

We define $\varrho \upharpoonright \gamma$ by induction on $\gamma < \kappa^{+(n+1)}$, and in doing so, we make sure that conditions \ref{condition:general1}-\ref{condition:general4} are satisfied by $\varrho \upharpoonright \gamma$. For clarity, when we say ``$\varrho \upharpoonright \gamma$ satisfies \ref{condition:general1}-\ref{condition:general4}'' we mean that the map $\varrho \upharpoonright \gamma$ satisfies statements \ref{condition:general1}-\ref{condition:general4} for all $\alpha < \gamma$ and $p \in B$.

We fix a well-ordering $\lessdot$ on a sufficiently large set, which \fxnote*{}{will ensure} uniformity in our choices during the inductive construction.

If $\lambda$ is a limit and we have defined $\varrho \upharpoonright \gamma$ so to satisfy \ref{condition:general1}-\ref{condition:general4} for every $\gamma < \lambda$, then clearly $\varrho \upharpoonright \lambda = \bigcup_{\gamma < \lambda} \varrho \upharpoonright \gamma$ still satisfies \ref{condition:general1}-\ref{condition:general4}. Thus, let us take care of the successor case. Suppose that we have defined $\varrho$ on $[\gamma]^2$, towards extending it to $[\gamma+1]^2$. \fxnote*{}{Recall our convention $\varrho(\alpha, \alpha) = \mathbf{0}_B$}. There are two cases: 
\begin{description}

\item[Case 1] $\mathrm{cf}(\gamma) < \mathrm{cf}(\kappa)$:  Fix an increasing sequence $\langle \gamma_\nu \mid \nu < \mathrm{cf}(\gamma) \rangle$ cofinal in $\gamma$. Let $p_\gamma$ be an upper bound in $B$ of the set
\[
 \big\{\varrho(\gamma_\mu, \gamma_\nu) \mid \mu \le \nu < \mathrm{cf}(\gamma)\big\}
\]
such that $|{\downarrow} p_\gamma| \ge \mathrm{cf}(\gamma)$. Note that such an upper bound always exists \fxnote*{}{because $B$ is $\mathrm{cf}(\kappa)$-directed by Lemma~\ref{lemma:breadth} and Proposition~\ref{prop:directed} and, moreover, $|B| = \kappa^{+n} \ge \mathrm{cf}(\kappa) > \mathrm{cf}(\gamma)$. First set $\varrho(0, \gamma) = \mathbf{0}_B$. Then, for each $\alpha$ with $0 < \alpha < \gamma$,} let $\nu < \mathrm{cf}(\gamma)$ be the least \fxnote*{}{ordinal} such that $\alpha \le \gamma_\nu$ and set $\varrho(\alpha, \gamma) = p_\gamma \vee \varrho(\alpha, \gamma_\nu)$.

\item[Case 2] $\mathrm{cf}(\gamma) \ge \mathrm{cf}(\kappa)$: Let $\vec{\theta}_\gamma=\langle \theta_\gamma(\nu) \mid \nu < \lambda_\gamma\rangle$ be the increasing enumeration of $C_\gamma$. We inductively define a sequence $\vec{p}_\gamma= \langle p_\gamma (\nu) \mid \nu < \lambda_\gamma \rangle$ of elements of $B$ as follows:
\begin{enumerate}[label={\upshape (\roman*)}]
\itemsep0.3em
\item Let $p_\gamma(0) = \mathbf{0}_B$.

\item If $\mathrm{cf}(\nu) < \mathrm{cf}(\kappa)$, then let $\langle \nu_\iota \mid \iota < \mathrm{cf}(\nu)\rangle$ be the $\lessdot$-least increasing sequence cofinal in $\nu$ and let $p_\gamma(\nu)$ be the $\lessdot$-least $p \in B$ such that $p$ is an upper bound in $B$ of the set 
\begin{equation}\label{eq:generalcase2}
\big\{p_\gamma(\nu_\iota) \vee \varrho(\theta_\gamma(\nu_\iota),\theta_\gamma(\nu))\mid \iota < \mathrm{cf}(\nu)\big\},
\end{equation}
and such that $p \neq p_\gamma(\mu)$ for every $\mu < \nu$. Such a $p$ always exists: as $B$ is $\mathrm{cf}(\kappa)$-directed, there exists $p' \in B$ which is an upper bound of the set \eqref{eq:generalcase2}; moreover, since ${\uparrow} p'$ has cardinality $\kappa^{+n}$, and since $|\nu| < \kappa^{+n}$, we can find a $p \in B$ such that $p' \le p$ and $p \neq p_\gamma(\mu)$ for every $\mu < \nu$.

\item If $\mathrm{cf}(\nu) \ge \mathrm{cf}(\kappa)$, then let $p_\gamma(\nu)$ be the $\lessdot$-least $p \in B$ such that $p \neq p_\gamma(\mu)$ for all  $\mu < \nu$. \vspace{0.3em}
\end{enumerate}

Now that we have defined the sequence $\vec{p}_\gamma$, we can extend $\varrho$: for each $\alpha < \gamma$, let $\mu_\gamma(\alpha)$ be the least $\nu < \lambda_\gamma$ such that $\alpha \le \theta_\gamma(\nu)$ and set $\varrho(\alpha, \gamma) = p_\gamma(\mu_\gamma(\alpha)) \vee \varrho(\alpha, \theta_\gamma(\mu_\gamma(\alpha)))$.
\end{description}

We assume that $\varrho \upharpoonright \gamma$ satisfies \ref{condition:general1}-\ref{condition:general4} \fxnote*{}{and has been defined following the inductive construction described above}. The rest of the proof consists of showing that $\varrho \upharpoonright (\gamma+1)$ also satisfies \ref{condition:general1}-\ref{condition:general4}.

\fxnote*{}{
\begin{claim}\label{claim:general0}
$\varrho(0, \gamma) = \mathbf{0}_B$.
\end{claim}
\begin{proof}
If $\mathrm{cf}(\gamma) < \mathrm{cf}(\kappa)$, then the definition of $\varrho \upharpoonright (\gamma+1)$ directly implies our claim. Now assume $\mathrm{cf}(\gamma) \ge \mathrm{cf}(\kappa)$. Clearly, $\mu_\gamma(0) = 0$; therefore, by definition of $\varrho \upharpoonright (\gamma+1)$, $\varrho(0, \gamma) = p_\gamma(0) \vee \varrho(0, \theta_\gamma(0))$. By construction, $p_\gamma(0) = \mathbf{0}_B$ and, by inductive assumption, $\varrho(0, \theta_\gamma(0)) = \mathbf{0}_B$. Hence, our claim follows.
\end{proof}
}

\fxnote*{}{
The next claim extracts from our square principle exactly what is needed for our proof.
\begin{claim}\label{claim:general0.5}
Suppose that $\mathrm{cf}(\gamma) \ge \mathrm{cf}(\kappa)$. Then, for every $\nu < \lambda_\gamma$ with $\mathrm{cf}(\nu) \ge \mathrm{cf}(\kappa)$, $\vec{\theta}_{\theta_\gamma(\nu)} = \vec{\theta}_\gamma \upharpoonright \nu$ and $\vec{p}_{\theta_\gamma(\nu)} = \vec{p}_\gamma \upharpoonright \nu$.
\end{claim}
\begin{proof}
It follows directly from our assumption $\mathrm{cf}(\nu) \ge \mathrm{cf}(\kappa)$ and from the closedness of $C_\gamma$ that $\theta_\gamma(\nu)$ is a limit point of $C_\gamma$ and $\mathrm{cf}(\theta_\gamma(\nu)) = \mathrm{cf}(\nu) \ge \mathrm{cf}(\kappa)$; by the coherence property of our square sequence, $C_{\theta_\gamma(\nu)} = C_\gamma \cap \theta_\gamma(\nu)$, and thus $\vec{\theta}_{\theta_\gamma(\nu)} = \vec{\theta}_\gamma \upharpoonright \nu$. We now claim that $\vec{p}_{\theta_\gamma(\nu)} = \vec{p}_\gamma \upharpoonright \nu$. Note the following: fix $\alpha\le\gamma$ with $\mathrm{cf}(\alpha) \ge \mathrm{cf}(\kappa)$ and $\mu < \lambda_\alpha$; then, according to the inductive definition of $\vec{p}_\alpha$ expressed by (i)--(iii) in Case 2 above, the initial segment $\vec{p}_\alpha \upharpoonright \mu$ is uniquely determined by the sequence $\vec{\theta}_{\alpha} \upharpoonright \mu$ and by the map $\varrho \upharpoonright \theta_\alpha(\mu)$, besides other trivially fixed parameters (among which there is the well-ordering $\lessdot$). In particular, since $\vec{\theta}_{\theta_\gamma(\nu)} = \vec{\theta}_\gamma \upharpoonright \nu$, we conclude that the two sequences $\vec{p}_{\theta_\gamma(\nu)}$ and $\vec{p}_\gamma \upharpoonright \nu$ coincide, as they satisfy the same inductive definition with the same parameters.
\end{proof}
}

\begin{claim}\label{claim:general1}
Suppose that $\mathrm{cf}(\gamma) \ge \mathrm{cf}(\kappa)$.  Then $\varrho(\theta_\gamma(\mu), \theta_\gamma(\nu)) \le p_\gamma(\mu) \vee p_\gamma(\nu)$ for every  $\mu < \nu < \lambda_\gamma$.
\end{claim} 
\begin{proof}
We prove our claim by induction on $\nu$. It vacuously holds when $\nu = 0$. 

Suppose $\nu > 0$ and $\mathrm{cf}(\nu) < \mathrm{cf}(\kappa)$. Pick any $\iota < \mathrm{cf}(\nu)$  such that $\mu \le \nu_\iota$.  By transitivity of $\varrho \upharpoonright \gamma$, \[\varrho(\theta_\gamma(\mu), \theta_\gamma(\nu)) \le \varrho(\theta_\gamma(\mu), \theta_\gamma(\nu_\iota)) \vee \varrho(\theta_\gamma(\nu_\iota), \theta_\gamma(\nu)).\]  By induction hypothesis, $\varrho(\theta_\gamma(\mu), \theta_\gamma(\nu_\iota)) \le p_\gamma(\mu) \vee p_\gamma(\nu_\iota)$.  Moreover, by definition of $p_\gamma(\nu)$ (see \eqref{eq:generalcase2}), \[p_\gamma(\nu_\iota) \vee \varrho(\theta_\gamma(\nu_\iota), \theta_\gamma(\nu)) \le p_\gamma(\nu).\] Combining these observations, we get $\varrho(\theta_\gamma(\mu), \theta_\gamma(\nu)) \le p_\gamma(\mu) \vee p_\gamma(\nu)$.

Finally, suppose that $\mathrm{cf}(\nu) \ge \mathrm{cf}(\kappa)$. 
The following holds:
\begin{equation}\label{eq:claim:general1}
\varrho(\theta_\gamma(\mu), \theta_\gamma(\nu)) = \varrho(\theta_{\theta_\gamma(\nu)}(\mu), \theta_\gamma(\nu)) = p_{\theta_\gamma(\nu)}(\mu) = p_\gamma(\mu),
\end{equation}
where the first and last equalities follow from \fxnote*{}{Claim~\ref{claim:general0.5}}, and the middle one comes directly from the definition of $\varrho \upharpoonright (\theta_\gamma(\nu)+1)$.
It follows from \eqref{eq:claim:general1} that $\varrho(\theta_\gamma(\mu), \theta_\gamma(\nu)) = p_\gamma(\mu) \le p_\gamma(\mu) \vee p_\gamma(\nu)$.
\end{proof}

\begin{claim}\label{claim:general2}
Suppose that $\mathrm{cf}(\gamma) \ge \mathrm{cf}(\kappa)$. For every $\alpha < \gamma$ and every $\nu$ with $\mu_\gamma(\alpha) \le \nu < \lambda_\gamma$, we have $p_\gamma(\mu_\gamma(\alpha)) \le p_\gamma(\nu) \vee \varrho(\alpha, \theta_\gamma(\nu))$.
\end{claim}
\begin{proof}
We prove the claim by induction on $\nu$. Clearly, the claim holds when $\nu = \mu_\gamma(\alpha)$. 

Suppose that $\mu_\gamma(\alpha) < \nu$ and $\mathrm{cf}(\nu) < \mathrm{cf}(\kappa)$. Fix some $\iota < \mathrm{cf}(\nu)$ such that $\mu_\gamma(\alpha) \le \nu_\iota$. By induction hypothesis, $p_\gamma(\mu_\gamma(\alpha)) \le p_\gamma(\nu_\iota) \vee \varrho(\alpha, \theta_\gamma(\nu_\iota))$. By subadditivity of $\varrho \upharpoonright \gamma$, we have $\varrho(\alpha, \theta_\gamma(\nu_\iota)) \le \varrho(\alpha, \theta_\gamma(\nu)) \vee \varrho(\theta_\gamma(\nu_\iota), \theta_\gamma(\nu))$. But since, by definition, $p_\gamma(\nu_\iota) \vee \varrho(\theta_\gamma(\nu_\iota), \theta_\gamma(\nu)) \le p_\gamma(\nu)$, we conclude that $p_\gamma(\mu_\gamma(\alpha)) \le p_\gamma(\nu) \vee \varrho(\alpha, \theta_\gamma(\nu))$.

Finally, suppose that $\mu_\gamma(\alpha) < \nu$ and $\mathrm{cf}(\nu) \ge \mathrm{cf}(\kappa)$. \fxnote*{}{The following holds:
\begin{align*}
\varrho(\alpha, \theta_\gamma(\nu)) &= p_{\theta_\gamma(\nu)}(\mu_{\theta_\gamma(\nu)}(\alpha)) \vee \varrho(\alpha, \theta_{\theta_\gamma(\nu)}(\mu_{\theta_\gamma(\nu)}(\alpha)))\\
&= p_\gamma(\mu_\gamma(\alpha)) \vee \varrho(\alpha, \theta_\gamma(\mu_\gamma(\alpha))),
\end{align*}
where the first equality follows directly from the definition of $\varrho \upharpoonright \nobreak (\theta_\gamma(\nu)+\nobreak 1)$, and the second one is a direct consequence of Claim~\ref{claim:general0.5}---by looking at the definition of $\mu_\gamma(\alpha)$ at the end of Case 2, note that $\mu_{\theta_\gamma(\nu)}(\alpha) = \mu_\gamma(\alpha)$ directly follows from $\vec{\theta}_{\theta_\gamma(\nu)} = \vec{\theta}_\gamma \upharpoonright \nu$. }
Thus $p_\gamma(\mu_\gamma(\alpha)) \le \varrho(\alpha, \theta_\gamma(\nu))$ and, \textit{a fortiori}, $p_\gamma(\mu_\gamma(\alpha)) \le p_\gamma(\nu) \vee \varrho(\alpha, \theta_\gamma(\nu))$.
\end{proof}

\begin{claim}\label{claim:general3}
$\varrho \upharpoonright (\gamma+1)$ is transitive.
\end{claim}
\begin{proof}
Since we are assuming that $\varrho \upharpoonright \gamma$ is transitive, it suffices to show that $\varrho(\alpha, \gamma) \le \varrho(\alpha, \beta) \vee \varrho(\beta, \gamma)$ for every $\alpha, \beta$ with $\alpha < \beta < \gamma$. \fxnote*{}{The case $\alpha = 0$ is trivial, as $\varrho(0, \gamma) = \mathbf{0}_B$ by Claim~\ref{claim:general0}.} Hence, fix $\alpha,\beta$ with $0 < \alpha < \beta < \gamma$.

First suppose that $\mathrm{cf}(\gamma) < \mathrm{cf}(\kappa)$. Let $\mu, \nu$ be the least such that $\alpha \le \gamma_\mu$ and $\beta \le \gamma_\nu$, respectively.  Clearly, $\mu \le \nu$. If $\mu = \nu$, our claim follows straightforwardly from the transitivity of $\varrho \upharpoonright \gamma$. Indeed we would have $\varrho(\alpha, \gamma_{\mu}) \le \varrho(\alpha, \beta) \vee \varrho(\beta, \gamma_{\mu}) = \varrho(\alpha, \beta) \vee \varrho(\beta, \gamma_{\nu})$, and thus 
\[
\varrho(\alpha, \gamma) = p_\gamma \vee \varrho(\alpha, \gamma_{\mu}) \le p_\gamma \vee \varrho(\alpha, \beta) \vee \varrho(\beta, \gamma_{\nu}) = \varrho(\alpha, \beta) \vee \varrho(\beta, \gamma),
\]
where the equalities hold by definition of $\varrho \upharpoonright (\gamma +1)$. Hence suppose that $\mu < \nu$. In this case $\gamma_{\mu} < \beta$ must hold by the minimality of $\nu$. The following holds:
\begin{align*}
\varrho(\alpha, \gamma) &= p_\gamma \vee \varrho(\alpha, \gamma_\mu)\\
&\le p_\gamma \vee \varrho(\alpha, \beta) \vee \varrho(\gamma_\mu, \beta)\\
&\le p_\gamma \vee \varrho(\alpha, \beta) \vee \varrho(\gamma_\mu, \gamma_\nu) \vee \varrho(\beta, \gamma_\nu)\\
&= p_\gamma \vee \varrho(\alpha, \beta) \vee\varrho(\beta, \gamma_\nu)\\
&= \varrho(\alpha, \beta) \vee \varrho(\beta, \gamma),
\end{align*}
where the first and last equality hold by definition of $\varrho \upharpoonright (\gamma+1)$; the two inequalities follow from the subadditivity of $\varrho \upharpoonright \gamma$ and, lastly, the second equality holds because $\varrho(\gamma_\mu, \gamma_\nu) \le p_\gamma$ by definition of $p_\gamma$.

Now suppose that $\mathrm{cf}(\gamma) \ge \mathrm{cf}(\kappa)$ and $\mu_\gamma(\alpha) = \mu_\gamma(\beta)$. By the transitivity of $\varrho \upharpoonright \gamma$ we have
\[
\varrho(\alpha, \theta_\gamma(\mu_\gamma(\alpha))) \le \varrho(\alpha, \beta) \vee \varrho(\beta, \theta_\gamma(\mu_\gamma(\alpha)))
\]
and thus
\begin{align*}
\varrho(\alpha, \gamma) &= p_\gamma(\mu_\gamma(\alpha)) \vee \varrho(\alpha, \theta_\gamma(\mu_\gamma(\alpha)))\\
&\le p_\gamma(\mu_\gamma(\alpha)) \vee \varrho(\alpha, \beta) \vee \varrho(\beta, \theta_\gamma(\mu_\gamma(\alpha)))\\
&= \varrho(\alpha, \beta) \vee p_\gamma(\mu_\gamma(\beta)) \vee \varrho(\beta, \theta_\gamma(\mu_\gamma(\beta)))\\
&= \varrho(\alpha, \beta) \vee \varrho(\beta, \gamma),
\end{align*}
where the second equality comes from the hypothesis $\mu_\gamma(\alpha) = \mu_\gamma(\beta)$ and the other two equalities hold by definition of $\varrho \upharpoonright (\gamma+1)$.

Finally, suppose that $\mathrm{cf}(\gamma) \ge \mathrm{cf}(\kappa)$ and $\mu_\gamma(\alpha) < \mu_\gamma(\beta)$. Note that $\theta_\gamma(\mu_\gamma(\alpha))$ must be smaller than $\beta$ by the minimality of $\mu_\gamma(\beta)$. The following holds:
\begin{equation}\label{eq:general3-1}
\begin{split}
\varrho(\alpha, \gamma) &= p_\gamma(\mu_\gamma(\alpha)) \vee \varrho(\alpha, \theta_\gamma(\mu_\gamma(\alpha)))\\
&\le p_\gamma(\mu_\gamma(\alpha)) \vee \varrho(\alpha, \beta) \vee \varrho(\theta_\gamma(\mu_\gamma(\alpha)), \beta)\\
&\le  p_\gamma(\mu_\gamma(\alpha)) \vee \varrho(\alpha, \beta) \vee \varrho(\theta_\gamma(\mu_\gamma(\alpha)), \theta_\gamma(\mu_\gamma(\beta))) \vee \varrho(\beta, \theta_\gamma(\mu_\gamma(\beta)))\\
&\le p_\gamma(\mu_\gamma(\alpha)) \vee \varrho(\alpha, \beta) \vee p_\gamma(\mu_\gamma(\beta)) \vee \varrho(\beta, \theta_\gamma(\mu_\gamma(\beta)))\\
&= p_\gamma(\mu_\gamma(\alpha)) \vee \varrho(\alpha, \beta) \vee \varrho(\beta,\gamma),
\end{split}
\end{equation}
where the first two inequalities hold because of the subadditivity of $\varrho \upharpoonright \gamma$, while the last inequality holds by Claim~\ref{claim:general1}. To finish the argument, note the following:
\begin{equation}\label{eq:general3-2}
\begin{split}
p_\gamma(\mu_\gamma(\alpha)) &\le p_\gamma(\mu_\gamma(\beta)) \vee \varrho(\alpha, \theta_\gamma(\mu_\gamma(\beta)))\\
&\le p_\gamma(\mu_\gamma(\beta)) \vee \varrho(\alpha, \beta) \vee \varrho(\beta, \theta_\gamma(\mu_\gamma(\beta)))\\
&= \varrho(\alpha, \beta) \vee \varrho(\beta,\gamma),
\end{split}
\end{equation}
where the first inequality holds by Claim~\ref{claim:general2} and the second inequality follows from the transitivity of $\varrho \upharpoonright \gamma$. By combining \eqref{eq:general3-1} and \eqref{eq:general3-2} we get $\varrho(\alpha, \gamma) \le \varrho(\alpha, \beta) \vee\nobreak \varrho(\beta, \gamma)$.
\end{proof}

\begin{claim}\label{claim:general4}
$\varrho \upharpoonright (\gamma+1)$ is subadditive.
\end{claim}
\begin{proof}
Since we are assuming that $\varrho \upharpoonright \gamma$ is subadditive, it suffices to show that $\varrho(\alpha, \beta) \le \varrho(\alpha, \gamma) \vee \varrho(\beta, \gamma)$ for every $\alpha, \beta$ with $\alpha < \beta < \gamma$. \fxnote*{}{The case $\alpha = 0$ is trivial, as $\varrho(0, \beta) = \mathbf{0}_B$ by induction hypothesis.} Hence fix some $\alpha, \beta$ with $0< \alpha < \beta < \gamma$.

First suppose that $\mathrm{cf}(\gamma) < \mathrm{cf}(\kappa)$. Let $\mu, \nu$ be the least such that $\alpha \le \gamma_\mu$ and $\beta \le \gamma_\nu$, respectively. If $\mu = \nu$, our claim follows from the subadditivity of $\varrho \upharpoonright \gamma$. Indeed we would have $\varrho(\alpha,\beta) \le \varrho(\alpha, \gamma_\mu) \vee \varrho(\beta, \gamma_\mu) = \varrho(\alpha, \gamma_\mu) \vee \varrho(\beta, \gamma_\nu)$ and, \textit{a fortiori},
\[
\varrho(\alpha,\beta) \le \varrho(\alpha, \gamma_\mu) \vee \varrho(\beta, \gamma_\nu) \vee p_\gamma = \varrho(\alpha, \gamma) \vee \varrho(\beta, \gamma),
\]
where the equality directly follows from the definition of $\varrho \upharpoonright (\gamma+1)$. Suppose now that $\mu < \nu$. In this case, $\gamma_{\mu}$  must be smaller than $\beta$ by the minimality of $\nu$. The following holds:
\begin{align*}
\varrho(\alpha, \beta) &\le \varrho(\alpha, \gamma_\mu) \vee \varrho(\gamma_\mu, \beta)\\
&\le \varrho(\alpha, \gamma_\mu) \vee \varrho(\gamma_\mu, \gamma_\nu) \vee \varrho(\beta, \gamma_\nu)\\
&\le \varrho(\alpha, \gamma_\mu) \vee p_\gamma \vee \varrho(\beta, \gamma_\nu)\\
&= \varrho(\alpha, \gamma) \vee \varrho(\beta, \gamma),
\end{align*}
where the first inequality holds by the transitivity of $\varrho \upharpoonright \gamma$, the second inequality instead holds by the subadditivity of $\varrho \upharpoonright \gamma$, and the last one follows from $\varrho(\gamma_\mu, \gamma_\nu) \le p_\gamma$, which holds by definition of $p_\gamma$.

Now suppose that $\mathrm{cf}(\gamma) \ge \mathrm{cf}(\kappa)$ and $\mu_\gamma(\alpha) = \mu_\gamma(\beta)$. By assumption and the subadditivity of $\varrho \upharpoonright \gamma$, 
\begin{align*}
\varrho(\alpha, \beta) &\le \varrho(\alpha, \theta_\gamma(\mu_\gamma(\alpha))) \vee \varrho(\beta, \theta_\gamma(\mu_\gamma(\alpha)))\\
&=\varrho(\alpha, \theta_\gamma(\mu_\gamma(\alpha))) \vee \varrho(\beta, \theta_\gamma(\mu_\gamma(\beta)))
\end{align*}
and, \textit{a fortiori},
\begin{align*}
\varrho(\alpha, \beta) &\le \varrho(\alpha, \theta_\gamma(\mu_\gamma(\alpha))) \vee \varrho(\beta, \theta_\gamma(\mu_\gamma(\beta))) \vee p_\gamma(\mu_\gamma(\alpha))\\
&= \varrho(\alpha, \gamma) \vee \varrho(\beta, \gamma).
\end{align*}

Finally, suppose $\mathrm{cf}(\gamma) \ge \mathrm{cf}(\kappa)$ and $\mu_\gamma(\alpha) < \mu_\gamma(\beta)$. Note that $\theta_\gamma(\mu_\gamma(\alpha))$ must be smaller than $\beta$ by the minimality of $\mu_\gamma(\beta)$. The following holds:
\begin{align*}
\varrho(\alpha, \beta) &\le \varrho(\alpha, \theta_\gamma(\mu_\gamma(\alpha))) \vee \varrho(\theta_\gamma(\mu_\gamma(\alpha)), \beta)\\
&\le \varrho(\alpha, \theta_\gamma(\mu_\gamma(\alpha))) \vee \varrho(\theta_\gamma(\mu_\gamma(\alpha)), \theta_\gamma(\mu_\gamma(\beta))) \vee \varrho(\beta, \theta_\gamma(\mu_\gamma(\beta)))\\
&\le \varrho(\alpha, \theta_\gamma(\mu_\gamma(\alpha))) \vee p_\gamma(\mu_\gamma(\alpha)) \vee p_\gamma(\mu_\gamma(\beta)) \vee \varrho(\beta, \theta_\gamma(\mu_\gamma(\beta)))\\
&= \varrho(\alpha, \gamma) \vee \varrho(\beta, \gamma),
\end{align*}
where the first inequality follows from the transitivity of $\varrho \upharpoonright \gamma$, the second one follows from the subadditivity of $\varrho \upharpoonright \gamma$, while the last holds by Claim~\ref{claim:general1}.
\end{proof}
\begin{claim}\label{claim:general5}
$|D_\varrho(\gamma, p)| \le |{\downarrow}p| + \aleph_0$ for every $p \in B$.
\end{claim}
\begin{proof}
Let us assume first that $\mathrm{cf}(\gamma) < \mathrm{cf}(\kappa)$. If $p_\gamma \not\le p$ then, by definition of $\varrho \upharpoonright (\gamma+1)$, \fxnote*{}{$D_\varrho(\gamma, p) = \{0\}$}. Thus, we can suppose $p_\gamma \le p$. \fxnote*{}{ The following is a direct consequence of the definition of $\varrho \upharpoonright (\gamma+1)$ and of  the fact that $\varrho \upharpoonright \gamma$ satisfies \ref{condition:general4} by induction hypothesis,
\begin{equation}\label{eq:general5-1}
D_\varrho(\gamma, p) \subseteq \bigcup \big\{D_\varrho(\gamma_\mu, p) \cup \{\gamma_\mu\} \mid \mu < \mathrm{cf}(\gamma)\big\}.
\end{equation}
} By induction hypothesis, $\varrho \upharpoonright \gamma$ satisfies \ref{condition:general2}, and therefore $|D_\varrho(\gamma_\mu, p)| \le |{\downarrow}p|+\nobreak \aleph_0$ for every $\mu < \mathrm{cf}(\gamma)$. This  observation, together with \eqref{eq:general5-1}, implies that $|D_\varrho(\gamma, p)| \le \mathrm{cf}(\gamma) \cdot |{\downarrow}p| + \aleph_0$. But since, by the definition of $p_\gamma$, $|{\downarrow}p_\gamma| \ge \mathrm{cf}(\gamma)$, we conclude that $|D_\varrho(\gamma, p)| \le |{\downarrow}p| + \aleph_0$.

Now assume $\mathrm{cf}(\gamma) \ge \mathrm{cf}(\kappa)$. \fxnote*{}{The following is a direct consequence of the definition of $\varrho \upharpoonright (\gamma+1)$:
\begin{equation}\label{eq:general5-2}
D_\varrho(\gamma, p) \subseteq \bigcup \big\{D_\varrho(\theta_\gamma(\mu), p) \cup \{\theta_\gamma(\mu)\} \mid \mu < \lambda_\gamma \text{ and } p_\gamma(\mu) \le p\big\}.
\end{equation}}
By induction hypothesis, $\varrho \upharpoonright \gamma$ satisfies \ref{condition:general2}, and therefore $|D_\varrho(\theta_\gamma(\mu), p)| \le |{\downarrow}p|+\nobreak \aleph_0$ for every $\mu < \lambda_\gamma$. This last observation, together with \eqref{eq:general5-2}, implies that 
\[
|D_\varrho(\gamma, p)| \le |\{\mu < \lambda_\gamma \mid p_\gamma(\mu) \le p\}| \cdot |{\downarrow}p| + \aleph_0.
\]
By construction, $\vec{p}_\gamma$ is injective, and therefore $|\{\mu < \lambda_\gamma \mid p_\gamma(\mu) \le p\}| \le |{\downarrow}p|$. Thus, we conclude that $|D_\varrho(\gamma, p)| \le |{\downarrow}p| + \aleph_0$.
\end{proof}
\fxnote*{}{We note that \eqref{eq:general5-1} and \eqref{eq:general5-2} can actually be proven to be equalities by an argument analogous to the one employed in Claim~\ref{claim:ladder8}.}
\begin{claim}\label{claim:general6}
$|D_\varrho(\gamma, p)| < \kappa$ for every $p \in B$.
\end{claim}
\begin{proof}
If $\kappa$ is uncountable, the result follows directly from Claim~\ref{claim:general5} and our hypotheses on $B$. So assume $\kappa = \aleph_0$. 

If $\gamma = \beta+1$ for some $\beta$, then, by \eqref{eq:general5-1}, \fxnote*{}{$D_\varrho(\gamma, p) \subseteq D_\varrho(\beta, p) \cup \{\beta\}$}. Since by induction hypothesis $D_\varrho(\beta, p)$ is finite, we conclude that also $D_\varrho(\gamma, p)$ is finite.

If $\gamma$ is a limit, then it follows from \eqref{eq:general5-2} and from the injectivity of $\vec{p}_\gamma$ that $D_\varrho(\gamma, p)$, being a finite union of finite sets, is finite. 
\end{proof}

\end{proof}

\section{Proof of Theorem~\ref{thm:main3}}\label{sec:thm3}

This section is devoted to the proof of the following theorem, which implies Theorem~\ref{thm:main3}.

\begin{theorem}\label{thm:mainplus3}
Let $\kappa$ be an infinite cardinal and $n\in \omega$. If $\square_{\kappa^{+m}}$ holds for every $m < n$, then there exists a special $(n+1, \kappa)$-ladder of cardinality $\kappa^{+n}$.
\end{theorem}
\begin{proof}

The proof's strategy is analogous to the one of Theorem~\ref{thm:mainplus2}. Moreover, since the case $\kappa = \aleph_0$ already follows from Theorem~\ref{thm:mainplus2}, we can suppose that $\kappa$ is uncountable.

We prove the result by induction on $n\in\omega$. The case $n=0$ is trivially true. So let us fix a special $(n+1, \kappa)$-ladder $(B, \le)$ of cardinality $\kappa^{+n}$ and a $\square_{\kappa^{+n}}$-sequence $\langle C_\alpha \mid \alpha < \kappa^{+(n+1)} \text{ and } \alpha \text{ limit}\rangle$, towards constructing a special $(n+2, \kappa)$-ladder of cardinality $\kappa^{+(n+1)}$. 

Let $P$ be the lattice such that $B$ is a quasi-product of $\kappa^{+n}$ and $P$. We want to define a map $\varrho: [\kappa^{+(n+1)}]^2\rightarrow B$ such that, for every $\alpha < \kappa^{+(n+1)}$ and every $p \in B$:

\begin{enumerate}[label=(\alph*)]
\itemsep0.3em
\item\label{condition:ladder1} $\varrho$ is \fxnote*{}{both} transitive and subadditive.
\item\label{condition:ladder2} $|D_\varrho(\alpha, p)| \le |{\downarrow}p|+ \aleph_0$.
\item\label{condition:ladder3} $D_\varrho(\alpha, p)$ is closed in $\alpha$.
\item\label{condition:ladder4} $\varrho(0, \alpha) = \mathbf{0}_B$.
\end{enumerate}

Note that conditions~\ref{condition:ladder3} and \ref{condition:ladder4} are new with respect to the conditions used in the proof of Theorem~\ref{thm:mainplus2}. With these new conditions, the same argument used at the beginning of the proof of Theorem~\ref{thm:mainplus2} yields that \fxnote*{}{$(\kappa^{+(n+1)} \times B, \trianglelefteq_\varrho)$} is a special $(n+2, \kappa)$-ladder.

We define $\varrho \upharpoonright \gamma$ by induction on $\gamma < \kappa^{+(n+1)}$, and in doing so, we make sure that conditions \ref{condition:ladder1}-\ref{condition:ladder4} are \fxnote*{}{satisfied} for $\varrho \upharpoonright \gamma$.  As in the proof of Theorem~\ref{thm:mainplus2}, when we say ``$\varrho \upharpoonright \gamma$ satisfies \ref{condition:ladder1}-\ref{condition:ladder4}'' we mean that the map $\varrho \upharpoonright \gamma$ satisfies statements \ref{condition:ladder1}-\ref{condition:ladder4} for all $\alpha < \gamma$.

Fix a well-ordering $\lessdot$ on a sufficiently large set, which \fxnote*{}{will ensure} uniformity in our choices during the inductive construction.

If $\lambda$ is a limit and we have defined $\varrho \upharpoonright \gamma$ so to satisfy \ref{condition:ladder1}-\ref{condition:ladder4} for every $\gamma < \lambda$, then clearly $\varrho \upharpoonright \lambda = \bigcup_{\gamma < \lambda} \varrho \upharpoonright \gamma$ still satisfies \ref{condition:ladder1}-\ref{condition:ladder4}. Thus, let us take care of the successor case. 

Suppose that we have defined $\varrho$ on $[\gamma]^2$, towards extending it to $[\gamma+1]^2$. \fxnote*{}{Recall our convention $\varrho(\alpha, \alpha) = \mathbf{0}_B$}. There are two cases: 
\begin{description}
\itemsep0.3em
\item[Case 1] $\gamma$ is a successor ordinal:  For each $\alpha < \gamma$, set $\varrho(\alpha, \gamma) = \varrho(\alpha, \gamma-1)$.

\item[Case 2] $\gamma$ is a limit ordinal: Let $\vec{\theta}_\gamma=\langle \theta_\gamma(\nu) \mid \nu < \lambda_\gamma\rangle$ be the increasing enumeration of $C_\gamma$. We inductively define a sequence $\vec{p}_\gamma= \langle p_\gamma (\nu) \mid \nu < \lambda_\gamma \rangle$ of elements of $B$ as follows:

\begin{enumerate}[label={\upshape (\roman*)}]
\itemsep0.3em
\item Let $p_\gamma(0) = \mathbf{0}_B = (0, \mathbf{0}_P)$.

\item If $\nu$ is a successor ordinal, let $p_\gamma(\nu)$ be the $\lessdot$-least $p \in B$ such that \fxnote*{}{$p_\gamma(\nu-1) \vee \varrho(\theta_\gamma(\nu-1), \theta_\gamma(\nu)) \le p$} and $\mathrm{ht}(p) > \mathrm{ht}(p_\gamma(\nu-1))$ and $p \neq p_\gamma(\mu)$ for all $\mu < \nu$. Note that such a $p$ always exists: first pick $p'$ such that \fxnote*{}{$p_\gamma(\nu-1) \vee \varrho(\theta_\gamma(\nu-1), \theta_\gamma(\nu)) \le p'$}; then, as $B$ is a quasi-product of $\kappa^{+n}$ and $P$, there exists $p'' > p'$ with $\mathrm{ht}(p'') > \mathrm{ht}(p')$, and, \textit{a fortiori}, $\mathrm{ht}(p'') > \mathrm{ht}(p_\gamma(\nu-1))$; finally, since $|{\uparrow}p''| = \kappa^{+n}$ and $|\nu| < \kappa^{+n}$, we conclude that there exists $p \ge p''$ such that $p \neq p_\gamma(\mu)$ for all $\mu < \nu$.

\item  If $\nu$ is a limit ordinal and\footnote{Given a sequence $(x_\mu)_{\mu < \nu}$, by $\liminf_{\mu < \nu} x_\mu$ we mean $\sup_{\mu < \nu} \inf_{\mu \le \xi < \nu} x_\xi$.} $\liminf_{\mu < \nu} \mathrm{ht}(p_\gamma(\mu)) < \kappa^{+n}$, then let
\[
p_\gamma(\nu) = \big(\liminf_{\mu < \nu} \mathrm{ht}(p_\gamma(\mu)), \mathbf{0}_P\big).
\]

\item\label{vectordef4} If $\nu$ is a limit ordinal and $\liminf_{\mu < \nu} \mathrm{ht}(p_\gamma(\mu)) = \kappa^{+n}$ (we will argue that this can happen only when $n=0$), then let $p_\gamma(\nu)$ be the $\lessdot$-least $p \in B$ such that $p \neq p_\gamma(\mu)$ for all $\mu < \nu$. \vspace{0.3em}
\end{enumerate}
Now that we have defined the sequence $\vec{p}_\gamma$, we can extend $\varrho$: for each $\alpha < \gamma$, let $\mu_\gamma(\alpha)$ be the least $\nu < \lambda_\gamma$ such that $\alpha \le \theta_\gamma(\nu)$, and set $\varrho(\alpha, \gamma) = p_\gamma(\mu_\gamma(\alpha)) \vee \varrho(\alpha, \theta_\gamma(\mu_\gamma(\alpha)))$.\vspace{0.3em}
\end{description}

We assume that $\varrho \upharpoonright \gamma$ satisfies \ref{condition:ladder1}-\ref{condition:ladder4} \fxnote*{}{and has been defined following the inductive construction described above}. The rest of the proof consists of showing that $\varrho \upharpoonright (\gamma+1)$ also satisfies \ref{condition:ladder1}-\ref{condition:ladder4}. If $\gamma$ is a limit, then for every $p \in B$ we let 
\[
\Delta_\gamma(p) \coloneqq \big\{\nu < \lambda_\gamma \mid p_\gamma(\nu) \le p\big\}.
\]

\begin{claim}
$\varrho(0, \gamma) = \mathbf{0}_B$.
\end{claim}
\begin{proof}
If $\gamma$ is a successor ordinal, then, by definition of $\varrho \upharpoonright (\gamma+1)$, $\varrho(0, \gamma) = \varrho(0, \gamma-1)$. Since we are assuming $\varrho(0, \gamma-1) = \mathbf{0}_B$, the claim follows. 

Now suppose that $\gamma$ is a limit. Clearly, $\mu_\gamma(0) = 0$; therefore, by definition of $\varrho \upharpoonright (\gamma+1)$, $\varrho(0, \gamma) = p_\gamma(0) \vee \varrho(0, \theta_\gamma(0))$. By construction, $p_\gamma(0) = \mathbf{0}_B$ and, by inductive assumption, $\varrho(0, \theta_\gamma(0)) = \mathbf{0}_B$. Hence, the claim follows.
\end{proof}

\begin{claim}\label{claim:ladder1}
Suppose that $\gamma$ is a limit. Then, for each $p \in B$, the set $\Delta_\gamma(p)$ is closed in $\lambda_\gamma$.
\end{claim} 
\begin{proof}

Pick any limit point $\nu$ of $\Delta_\gamma(p)$, towards showing that $\nu \in \Delta_\gamma(p)$. There must be cofinally many $\mu < \nu$ such that $p_\gamma(\mu) \le p$. In particular, 
\begin{equation}\label{eq:claim:ladder1-1}
\liminf_{\mu < \nu} \mathrm{ht}(p_\gamma(\mu)) \le \mathrm{ht}(p) < \kappa^{+n},
\end{equation}
and hence, by definition of $\vec{p}_\gamma$,
\begin{equation}\label{eq:claim:ladder1-2}
p_\gamma(\nu) = \big(\liminf_{\mu < \nu} \mathrm{ht}(p_\gamma(\mu)), \mathbf{0}_P\big).
\end{equation}

If $n = 0$, then $B$ can be identified with $\kappa$ (see the remark after Definition~\ref{def:specialladder}). In particular, the inequality \eqref{eq:claim:ladder1-1} becomes
\[
\liminf_{\mu < \nu} p_\gamma(\mu) \le p < \kappa,
\]
and \eqref{eq:claim:ladder1-2} becomes
\[
p_\gamma(\nu) = \liminf_{\mu < \nu} p_\gamma(\mu).
\]
Therefore, $p_\gamma(\nu) \le p$, or, equivalently, $\nu \in \Delta_\gamma(p)$.

Now suppose $n > 0$. Since $\kappa^{+n}$ is regular, \fxnote*{}{$\vec{p}_\gamma \upharpoonright \nu'$} does not satisfy the hypothesis of case~\ref{vectordef4} of the definition of $\vec{p}_\gamma$ for any \fxnote*{}{$\nu' < \lambda_\gamma$}. In other words, \fxnote*{}{$\liminf_{\mu < \nu'} \mathrm{ht}(p_\gamma(\mu)) < \kappa^{+n}$ for all $\nu' < \lambda_\gamma$}. It  follows from the definition of $\vec{p}_\gamma$ that the sequence $\mathrm{ht}\circ \vec{p}_\gamma = \langle \mathrm{ht}(p_\gamma(\mu)) \mid \mu < \lambda_\gamma\rangle$ is \fxnote*{}{strictly} increasing. In particular,
\begin{equation}\label{eq:claim:ladder1-3}
\begin{split}
\liminf_{\mu < \nu} \mathrm{ht}(p_\gamma(\mu)) &= \sup_{\mu < \nu} \mathrm{ht}(p_\gamma(\mu))\\
&= \sup\big\{\mathrm{ht}(p_\gamma(\mu)) \mid \mu < \nu \text{ and } p_\gamma(\mu) \le p\big\},
\end{split}
\end{equation}
where the first equality follows from the monotonicity of $\mathrm{ht} \circ \vec{p}_\gamma$ and the second one from the already noted fact that for cofinally many $\mu < \nu$, $p_\gamma(\mu) \le p$. Since $B$ is a lattice, then $\{\mathrm{ht}(q) \mid q \le p\}$ is a closed subset of $\kappa^{+n}$ \fxnote*{}{(see Remark~\ref{rmk:closedness})}. By this observation, by \eqref{eq:claim:ladder1-3} and \eqref{eq:claim:ladder1-2}, we conclude
\[
\mathrm{ht}(p_\gamma(\nu)) \in \{\mathrm{ht}(q) \mid q \le p\}.
\]
Equivalently,  there exists $z \in B$ such that $\mathrm{ht}(z) = \mathrm{ht}(p_\gamma(\nu))$ and $z \le p$. Hence, $p_\gamma(\nu) = (\mathrm{ht}(z), \mathbf{0}_P)$. As $(B, \le) = (\kappa^{+n}\times P, \le)$ is a quasi-product of $\kappa^{+n}$ and $P$, we conclude that $p_\gamma(\nu) \le z$, and therefore $p_\gamma(\nu) \le p$, or, equivalently, $ \nu \in \Delta_\gamma(p)$.
\end{proof}

\begin{claim}\label{claim:ladder2}
Suppose that $\gamma$ is a limit. Then, for each $p \in B$, $|\Delta_\gamma(p)| \le |{\downarrow}p|$.
\end{claim}
\begin{proof}
Fix some $p \in B$, towards showing that $|\Delta_\gamma(p)| \le |{\downarrow}p|$.  If $n > 0$, we have already noted in the proof of Claim~\ref{claim:ladder1} that $\mathrm{ht} \circ \vec{p}_\gamma$ is strictly increasing. In particular, $\vec{p}_\gamma$ is injective, and it directly follows that $|\Delta_\gamma(p)| \le |{\downarrow}p|$. 

We are left to deal with $n = 0$. Recall that in this case we identify $B$ with $\kappa$. Let
\[
\Delta_\gamma^-(p) \coloneqq \big\{\nu < \lambda_\gamma \mid p_\gamma(\nu) \le p \text{ and } (\nu \text{ successor or } \liminf_{\mu < \nu}p_\gamma(\mu) = \kappa)\big\}.
\]
We claim that $\Delta_\gamma(p) \subseteq \mathrm{cl}(\Delta_\gamma^-(p))$, where $\mathrm{cl}(\Delta_\gamma^-(p))$ is the closure of the set $\Delta_\gamma^-(p)$. This suffices to prove that $|\Delta_\gamma(p)|\le |{\downarrow}p|$, as $\vec{p}_\gamma \upharpoonright \Delta_\gamma^-(p)$ is injective by definition of $\vec{p}_\gamma$, and hence, if our claim is true, we have $|{\downarrow}p| \ge |\Delta_\gamma^-(p)| = |\mathrm{cl}(\Delta_\gamma^-(p))| \ge |\Delta_\gamma(p)|$.

We prove that $\Delta_\gamma(p) \cap \nu \subseteq \mathrm{cl}(\Delta_\gamma^-(p)) \cap \nu$ for all $\nu < \lambda_\gamma$ by induction on $\nu$.  As the limit case is trivial, we can fix $\nu \in \Delta_\gamma(p)$ and suppose that $\Delta_\gamma(p) \cap \nu \subseteq \mathrm{cl}(\Delta_\gamma^-(p)) \cap \nu$, towards showing that $\nu \in \mathrm{cl}(\Delta_\gamma^-(p))$. If $\nu$ is a successor ordinal or $\liminf_{\mu < \nu} p_\gamma(\mu) = \kappa$, then $\nu \in \Delta_\gamma^-(p)$ by definition; otherwise, $p_\gamma(\nu) = \liminf_{\mu < \nu}p_\gamma(\mu)$, and thus $\liminf_{\mu < \nu}p_\gamma(\mu) \le p$ since we are assuming $p_\gamma(\nu) \le p$. But note that
\begin{equation}
\liminf_{\mu < \nu} p_\gamma(\mu) = \adjustlimits \sup_{\mu < \nu} \min_{\mu \le \xi < \nu} p_\gamma(\xi),
\end{equation}
as $B = \kappa$ is well-ordered. Hence, there must be cofinally many $\mu < \nu$ such that $p_\gamma(\mu) \le p$. In other words, $\nu$ is a limit point of $\Delta_\gamma(p)$. Since we assumed $\Delta_\gamma(p) \cap \nu \subseteq \mathrm{cl}(\Delta_\gamma^-(p)) \cap \nu$, we conclude that $\nu$ belongs to $\mathrm{cl}(\Delta_\gamma^-(p))$, being a limit point of it. 
\end{proof}

Claims~\ref{claim:ladder2.5}--\ref{claim:ladder7} are just minor variations of Claims~\ref{claim:general0.5}--\ref{claim:general5}, respectively. 

\fxnote*{}{
\begin{claim}\label{claim:ladder2.5}
Suppose that $\gamma$ is a limit. Then, for every limit $\nu < \lambda_\gamma$, $\vec{\theta}_{\theta_\gamma(\nu)} = \vec{\theta}_\gamma \upharpoonright \nu$ and $\vec{p}_{\theta_\gamma(\nu)} = \vec{p}_\gamma \upharpoonright \nu$.
\end{claim}
\begin{proof}
The proof is the same, \textit{mutatis mutandis}, as that of Claim~\ref{claim:general0.5}.
\end{proof}
}

\begin{claim}\label{claim:ladder3}
Suppose that $\gamma$ is a limit.  Then $\varrho(\theta_\gamma(\mu), \theta_\gamma(\nu)) \le p_\gamma(\mu) \vee p_\gamma(\nu)$ for every  $\mu < \nu < \lambda_\gamma$.
\end{claim} 
\begin{proof}
We prove our claim by induction on $\nu$. It vacuously holds when $\nu = 0$. 

Suppose that $\nu$ is a successor ordinal. By transitivity of $\varrho \upharpoonright \gamma$, \[\varrho(\theta_\gamma(\mu), \theta_\gamma(\nu)) \le \varrho(\theta_\gamma(\mu), \theta_\gamma(\nu-1)) \vee \varrho(\theta_\gamma(\nu-1), \theta_\gamma(\nu)).\]  By induction hypothesis, $\varrho(\theta_\gamma(\mu), \theta_\gamma(\nu-1)) \le p_\gamma(\mu) \vee p_\gamma(\nu-1)$.  Moreover, $p_\gamma(\nu-\nobreak 1) \vee \varrho(\theta_\gamma(\nu-1), \theta_\gamma(\nu)) \le p_\gamma(\nu)$ by definition of $p_\gamma(\nu)$. Combining these observations, we get $\varrho(\theta_\gamma(\mu), \theta_\gamma(\nu)) \le p_\gamma(\mu) \vee p_\gamma(\nu)$.

Now suppose that $\nu$ is a limit. The following holds:
\begin{equation}\label{eq:claim:ladder3}
\varrho(\theta_\gamma(\mu), \theta_\gamma(\nu)) = \varrho(\theta_{\theta_\gamma(\nu)}(\mu), \theta_\gamma(\nu)) = p_{\theta_\gamma(\nu)}(\mu) = p_\gamma(\mu),
\end{equation}
where the first and last equalities follow from \fxnote*{}{Claim~\ref{claim:ladder2.5}}, and the middle one comes directly from the definition of $\varrho \upharpoonright (\theta_\gamma(\nu)+1)$.
It follows from \eqref{eq:claim:ladder3} that $\varrho(\theta_\gamma(\mu), \theta_\gamma(\nu)) = p_\gamma(\mu) \le p_\gamma(\mu) \vee p_\gamma(\nu)$.
\end{proof}

\begin{claim}\label{claim:ladder4}
Suppose that $\gamma$ is a limit. For every $\alpha < \gamma$ and every $\nu$ with $\mu_\gamma(\alpha) \le \nu < \lambda_\gamma$, we have $p_\gamma(\mu_\gamma(\alpha)) \le p_\gamma(\nu) \vee \varrho(\alpha, \theta_\gamma(\nu))$.
\end{claim}
\begin{proof}
We prove the claim by induction on $\nu$. Clearly, the claim holds when $\nu = \mu_\gamma(\alpha)$. 

Suppose that $\mu_\gamma(\alpha) < \nu$ and that $\nu$ is a successor ordinal. By induction hypothesis, $p_\gamma(\mu_\gamma(\alpha)) \le p_\gamma(\nu-1) \vee \varrho(\alpha, \theta_\gamma(\nu-1))$. By subadditivity of $\varrho \upharpoonright \gamma$, we have $\varrho(\alpha, \theta_\gamma(\nu-1)) \le \varrho(\alpha, \theta_\gamma(\nu)) \vee \varrho(\theta_\gamma(\nu-1), \theta_\gamma(\nu))$. But since, by definition, $p_\gamma(\nu-1) \vee \varrho(\theta_\gamma(\nu-1), \theta_\gamma(\nu)) \le p_\gamma(\nu)$, we conclude that $p_\gamma(\mu_\gamma(\alpha)) \le p_\gamma(\nu) \vee \varrho(\alpha, \theta_\gamma(\nu))$.

Finally, suppose that $\mu_\gamma(\alpha) < \nu$ and that $\nu$ is a limit.  \fxnote*{}{The following holds:
\begin{align*}
\varrho(\alpha, \theta_\gamma(\nu)) &= p_{\theta_\gamma(\nu)}(\mu_{\theta_\gamma(\nu)}(\alpha)) \vee \varrho(\alpha, \theta_{\theta_\gamma(\nu)}(\mu_{\theta_\gamma(\nu)}(\alpha)))\\
&= p_\gamma(\mu_\gamma(\alpha)) \vee \varrho(\alpha, \theta_\gamma(\mu_\gamma(\alpha))),
\end{align*}
where the first equality follows directly from the definition of $\varrho \upharpoonright \nobreak (\theta_\gamma(\nu)+\nobreak 1)$, and the second one is a direct consequence of Claim~\ref{claim:ladder2.5}---by looking at the definition of $\mu_\gamma(\alpha)$ at the end of Case 2, note that $\mu_{\theta_\gamma(\nu)}(\alpha) = \mu_\gamma(\alpha)$ directly follows from $\vec{\theta}_{\theta_\gamma(\nu)} = \vec{\theta}_\gamma \upharpoonright \nu$.}  
Thus $p_\gamma(\mu_\gamma(\alpha)) \le \varrho(\alpha, \theta_\gamma(\nu))$ and, \textit{a fortiori}, $p_\gamma(\mu_\gamma(\alpha)) \le p_\gamma(\nu) \vee \nobreak \varrho(\alpha, \theta_\gamma(\nu))$.
\end{proof}

\begin{claim}\label{claim:ladder5}
$\varrho \upharpoonright (\gamma+1)$ is transitive.
\end{claim}
\begin{proof}
Since we are assuming that $\varrho \upharpoonright \gamma$ is transitive, it suffices to show that $\varrho(\alpha, \gamma) \le \varrho(\alpha, \beta) \vee \varrho(\beta, \gamma)$ for every $\alpha, \beta$ with $\alpha < \beta < \gamma$. Hence, fix $\alpha,\beta$ with $\alpha < \beta < \gamma$.

First suppose that $\gamma$ is a successor ordinal. By the transitivity of $\varrho \upharpoonright \gamma$, $\varrho(\alpha, \gamma-\nobreak 1) \le \varrho(\alpha, \beta) \vee \varrho(\beta, \gamma-1)$, and thus 
\[
\varrho(\alpha, \gamma) = \varrho(\alpha, \gamma-1) \le \varrho(\alpha, \beta) \vee \varrho(\beta, \gamma-1) = \varrho(\alpha, \beta) \vee \varrho(\beta, \gamma),
\]
where the equalities hold by definition of $\varrho \upharpoonright (\gamma +1)$.

Now suppose that $\gamma$ is a limit and $\mu_\gamma(\alpha) = \mu_\gamma(\beta)$. By the transitivity of $\varrho \upharpoonright \gamma$ we have
\[
\varrho(\alpha, \theta_\gamma(\mu_\gamma(\alpha))) \le \varrho(\alpha, \beta) \vee \varrho(\beta, \theta_\gamma(\mu_\gamma(\alpha)))
\]
and thus
\begin{align*}
\varrho(\alpha, \gamma) &= p_\gamma(\mu_\gamma(\alpha)) \vee \varrho(\alpha, \theta_\gamma(\mu_\gamma(\alpha)))\\
&\le p_\gamma(\mu_\gamma(\alpha)) \vee \varrho(\alpha, \beta) \vee \varrho(\beta, \theta_\gamma(\mu_\gamma(\alpha)))\\
&= \varrho(\alpha, \beta) \vee p_\gamma(\mu_\gamma(\beta)) \vee \varrho(\beta, \theta_\gamma(\mu_\gamma(\beta)))\\
&= \varrho(\alpha, \beta) \vee \varrho(\beta, \gamma),
\end{align*}
where the second equality comes from the hypothesis $\mu_\gamma(\alpha) = \mu_\gamma(\beta)$ and the other two equalities hold by definition of $\varrho \upharpoonright (\gamma+1)$.

Finally, suppose that $\gamma$ is a limit and $\mu_\gamma(\alpha) < \mu_\gamma(\beta)$. Note that $\theta_\gamma(\mu_\gamma(\alpha))$ must be smaller than $\beta$ by the minimality of $\mu_\gamma(\beta)$. The following holds:
\begin{equation}\label{eq:ladder5-1}
\begin{split}
\varrho(\alpha, \gamma) &= p_\gamma(\mu_\gamma(\alpha)) \vee \varrho(\alpha, \theta_\gamma(\mu_\gamma(\alpha)))\\
&\le p_\gamma(\mu_\gamma(\alpha)) \vee \varrho(\alpha, \beta) \vee \varrho(\theta_\gamma(\mu_\gamma(\alpha)), \beta)\\
&\le  p_\gamma(\mu_\gamma(\alpha)) \vee \varrho(\alpha, \beta) \vee \varrho(\theta_\gamma(\mu_\gamma(\alpha)), \theta_\gamma(\mu_\gamma(\beta))) \vee \varrho(\beta, \theta_\gamma(\mu_\gamma(\beta)))\\
&\le p_\gamma(\mu_\gamma(\alpha)) \vee \varrho(\alpha, \beta) \vee p_\gamma(\mu_\gamma(\beta)) \vee \varrho(\beta, \theta_\gamma(\mu_\gamma(\beta)))\\
&= p_\gamma(\mu_\gamma(\alpha)) \vee \varrho(\alpha, \beta) \vee \varrho(\beta,\gamma),
\end{split}
\end{equation}
where the first two inequalities hold because of the subadditivity of $\varrho \upharpoonright \gamma$, while the last inequality holds by Claim~\ref{claim:ladder3}. To finish the argument, note the following:
\begin{equation}\label{eq:ladder5-2}
\begin{split}
p_\gamma(\mu_\gamma(\alpha)) &\le p_\gamma(\mu_\gamma(\beta)) \vee \varrho(\alpha, \theta_\gamma(\mu_\gamma(\beta)))\\
&\le p_\gamma(\mu_\gamma(\beta)) \vee \varrho(\alpha, \beta) \vee \varrho(\beta, \theta_\gamma(\mu_\gamma(\beta)))\\
&= \varrho(\alpha, \beta) \vee \varrho(\beta,\gamma),
\end{split}
\end{equation}
where the first inequality holds by Claim~\ref{claim:ladder4} and the second inequality follows from the transitivity of $\varrho \upharpoonright \gamma$. By combining \eqref{eq:ladder5-1} and \eqref{eq:ladder5-2} we get $\varrho(\alpha, \gamma) \le \varrho(\alpha, \beta) \vee \varrho(\beta, \gamma)$.
\end{proof}

\begin{claim}\label{claim:ladder6}
$\varrho \upharpoonright (\gamma+1)$ is subadditive.
\end{claim}
\begin{proof}
Since we are assuming that $\varrho \upharpoonright \gamma$ is subadditive, it suffices to show that $\varrho(\alpha, \beta) \le \varrho(\alpha, \gamma) \vee \varrho(\beta, \gamma)$ for every $\alpha, \beta$ with $\alpha < \beta < \gamma$. Hence fix some $\alpha, \beta$ with $\alpha < \beta < \gamma$.

First suppose that $\gamma$ is a successor ordinal. By the subadditivity of $\varrho \upharpoonright \gamma$, $\varrho(\alpha,\beta) \le \varrho(\alpha, \gamma-1) \vee \varrho(\beta, \gamma-1)$. It immediately follows by definition of $\varrho \upharpoonright (\gamma+1)$ that $\varrho(\alpha,\beta)  \le \varrho(\alpha, \gamma) \vee \varrho(\beta, \gamma)$.

Now suppose that $\gamma$ is a limit and $\mu_\gamma(\alpha) = \mu_\gamma(\beta)$. By assumption and the subadditivity of $\varrho \upharpoonright \gamma$, 
\begin{align*}
\varrho(\alpha, \beta) &\le \varrho(\alpha, \theta_\gamma(\mu_\gamma(\alpha))) \vee \varrho(\beta, \theta_\gamma(\mu_\gamma(\alpha)))\\
&=\varrho(\alpha, \theta_\gamma(\mu_\gamma(\alpha))) \vee \varrho(\beta, \theta_\gamma(\mu_\gamma(\beta)))
\end{align*}
and \textit{a fortiori} 
\begin{align*}
\varrho(\alpha, \beta) &\le \varrho(\alpha, \theta_\gamma(\mu_\gamma(\alpha))) \vee \varrho(\beta, \theta_\gamma(\mu_\gamma(\beta))) \vee p_\gamma(\mu_\gamma(\alpha))\\
&= \varrho(\alpha, \gamma) \vee \varrho(\beta, \gamma).
\end{align*}

Finally, suppose $\gamma$ is a limit and $\mu_\gamma(\alpha) < \mu_\gamma(\beta)$. Note that $\theta_\gamma(\mu_\gamma(\alpha))$ must be smaller than $\beta$ by the minimality of $\mu_\gamma(\beta)$. The following holds:
\begin{align*}
\varrho(\alpha, \beta) &\le \varrho(\alpha, \theta_\gamma(\mu_\gamma(\alpha))) \vee \varrho(\theta_\gamma(\mu_\gamma(\alpha)), \beta)\\
&\le \varrho(\alpha, \theta_\gamma(\mu_\gamma(\alpha))) \vee \varrho(\theta_\gamma(\mu_\gamma(\alpha)), \theta_\gamma(\mu_\gamma(\beta))) \vee \varrho(\beta, \theta_\gamma(\mu_\gamma(\beta)))\\
&\le \varrho(\alpha, \theta_\gamma(\mu_\gamma(\alpha))) \vee p_\gamma(\mu_\gamma(\alpha)) \vee p_\gamma(\mu_\gamma(\beta)) \vee \varrho(\beta, \theta_\gamma(\mu_\gamma(\beta)))\\
&= \varrho(\alpha, \gamma) \vee \varrho(\beta, \gamma),
\end{align*}
where the first inequality follows from the transitivity of $\varrho \upharpoonright \gamma$, the second one follows from the subadditivity of $\varrho \upharpoonright \gamma$, while the last holds by Claim~\ref{claim:ladder3}.
\end{proof}

\begin{claim}\label{claim:ladder7}
$|D_\varrho(\gamma, p)| \le |{\downarrow}p| + \aleph_0$ for every $p \in B$.
\end{claim}
\begin{proof}
Let us assume first that $\gamma$ is a successor ordinal. By definition of $\varrho \upharpoonright (\gamma+1)$, $D_\varrho(\gamma, p) = D_\varrho(\gamma-1, p) \cup \{\gamma-1\}$, and therefore the claim follows by induction hypothesis.

Now assume $\gamma$ is a limit. \fxnote*{}{The following is a direct consequence of the definition of $\varrho \upharpoonright (\gamma+1)$:
\begin{equation}\label{eq:claim:ladder7}
D_\varrho(\gamma, p) \subseteq \bigcup \big\{D_\varrho(\theta_\gamma(\mu), p) \cup \{\theta_\gamma(\mu)\} \mid \mu < \lambda_\gamma \text{ and } p_\gamma(\mu) \le p\big\}.
\end{equation}} 
By induction hypothesis, $\varrho \upharpoonright \gamma$ satisfies \ref{condition:ladder2}, and therefore $|D_\varrho(\theta_\gamma(\mu), p)| \le |{\downarrow}p|+\nobreak \aleph_0$ for every $\mu < \lambda_\gamma$. This last observation, together with \eqref{eq:claim:ladder7}, implies that 
\[
|D_\varrho(\gamma, p)| \le |\Delta_\gamma(p)| \cdot |{\downarrow}p| + \aleph_0.
\]
By Claim~\ref{claim:ladder2}, $|\Delta_\gamma(p)| \le |{\downarrow}p|$. Thus, we conclude that $|D_\varrho(\gamma, p)| \le |{\downarrow}p| + \aleph_0$.
\end{proof}

\begin{claim}\label{claim:ladder8}
$D_\varrho(\gamma, p)$ is closed in $\gamma$ for every $p \in B$.
\end{claim}
\begin{proof}
If $\gamma$ is a successor ordinal, then, by definition of $\varrho \upharpoonright (\gamma+1)$, $D_\varrho(\gamma, p) = D_\varrho(\gamma-1, p) \cup \{\gamma-1\}$. By induction hypothesis, $D_\varrho(\gamma-1, p)$ is closed in $\gamma-1$, and therefore $D_\varrho(\gamma, p)$ is closed in $\gamma$.

Now suppose that $\gamma$ is a limit. We first claim that (cf. \eqref{eq:claim:ladder7})
\fxnote*{}{
\begin{equation}\label{eq:claim:ladder7.5}
D_\varrho(\gamma, p) = \bigcup \big\{D_\varrho(\theta_\gamma(\mu), p) \cup \{\theta_\gamma(\mu)\} \mid \mu < \lambda_\gamma \text{ and } p_\gamma(\mu) \le p\big\}.
\end{equation}
One inclusion comes from \eqref{eq:claim:ladder7}. Now we focus on the other inclusion. Since, by definition, $\varrho(\theta_\gamma(\mu), \gamma) = p_\gamma(\mu)$ for every $\mu < \lambda_\gamma$, we automatically have that $\theta_\gamma(\mu) \in D_\varrho(\gamma, p)$ for every $\mu\in \Delta_\gamma(p)$. Now pick some $\mu \in \Delta_\gamma(p)$ and an $\alpha$ such that $\alpha \in D_\varrho(\theta_\gamma(\mu), p)$, towards showing that $\alpha \in D_\varrho(\gamma, p)$. By transitivity of $\varrho \upharpoonright (\gamma+1)$ (proved in Claim~\ref{claim:ladder5}), we have that
\[
\varrho(\alpha, \gamma) \le \varrho(\alpha, \theta_\gamma(\mu)) \vee \varrho(\theta_\gamma(\mu), \gamma) = \varrho(\alpha, \theta_\gamma(\mu)) \vee p_\gamma(\mu).
\]
Since we picked $\mu$ such that $p_\gamma(\mu) \le p$ and, by assumption, $\varrho(\alpha, \theta_\gamma(\mu)) \le p$, we conclude that $\varrho(\alpha, \gamma) \le p$, or, equivalently, $\alpha \in D_\varrho(\gamma, p)$. Thus, \eqref{eq:claim:ladder7.5} holds. 
}

We now claim that
\begin{equation}\label{eq:claim:ladder8}
\forall \mu \in \Delta_\gamma(p), \ D_\varrho(\gamma, p) \cap \theta_\gamma(\mu) = D_\varrho(\theta_\gamma(\mu), p).
\end{equation}
Towards showing that \eqref{eq:claim:ladder8} holds, fix some $\mu \in \Delta_\gamma(p)$. By \eqref{eq:claim:ladder7.5}, $D_\varrho(\theta_\gamma(\mu), p) \subseteq D_\varrho(\gamma, p)$. Now, given some $\alpha \in D_\varrho(\gamma, p) \cap \theta_\gamma(\mu)$, let us show that $\alpha$ belongs to $D_\varrho(\theta_\gamma(\mu), p)$. By \eqref{eq:claim:ladder7.5} again, there exists $\nu \in \Delta_\gamma(p)$ with $\alpha \in D_\varrho(\theta_\gamma(\nu), p) \cup \nobreak \{\theta_\gamma(\nu)\}$. If $\mu \le \nu$, then, by subadditivity of $\varrho\upharpoonright \gamma$,
\[
\varrho(\alpha, \theta_\gamma(\mu)) \le \varrho(\alpha, \theta_\gamma(\nu)) \vee \varrho(\theta_\gamma(\mu), \theta_\gamma(\nu));
\]
if $\nu < \mu$ instead, then, by transitivity of $\varrho\upharpoonright \gamma$,
\[
\varrho(\alpha, \theta_\gamma(\mu)) \le \varrho(\alpha, \theta_\gamma(\nu)) \vee \varrho(\theta_\gamma(\nu), \theta_\gamma(\mu)).
\]
In either case, we may conclude, by Claim~\ref{claim:ladder3}, that
\[
\varrho(\alpha, \theta_\gamma(\mu)) \le \varrho(\alpha, \theta_\gamma(\nu)) \vee p_\gamma(\mu) \vee p_\gamma(\nu).
\]
Therefore $\varrho(\alpha, \theta_\gamma(\mu)) \le p$, since both $\mu$ and $\nu$ belong to $\Delta_\gamma(p)$ and since we are assuming $\varrho(\alpha, \theta_\gamma(\nu)) \le p$. Thus, $\alpha \in D_\varrho(\theta_\gamma(\mu), p)$ as we wanted to show.

We are ready to prove that $D_\varrho(\gamma, p)$ is closed in $\gamma$. Pick $\alpha < \gamma$ such that $\alpha$ is a limit point of $D_\varrho(\gamma, p)$, towards showing that $\alpha \in D_\varrho(\gamma, p)$. There are two cases: either $\alpha$ is a limit point of $\{\theta_\gamma(\mu) \mid \mu \in \Delta_\gamma(p)\}$ or it is not. In the first case, by the \fxnote*{}{closedness} of $\Delta_\gamma(p)$ (Claim~\ref{claim:ladder1}) and the continuity of $\vec{\theta}_\gamma$, there exists $\mu \in \Delta_\gamma(p)$ such that $\alpha = \theta_\gamma(\mu)$; but also in the second case there must be some $\mu \in \Delta_\gamma(p)$ such that $\alpha \le \theta_\gamma(\mu)$, as otherwise, by \eqref{eq:claim:ladder7.5}, $D_\varrho(\gamma, p)$ would be bounded below $\alpha$, contradicting $\alpha$ being a limit point of $D_\varrho(\gamma, p)$. Therefore, in either case, there exists $\mu \in \Delta_\gamma(p)$ such that $\alpha \le \theta_\gamma(\mu)$. Fix one such $\mu$. If $\alpha = \theta_\gamma(\mu)$, then $\alpha \in D_\varrho(\gamma, p)$ by \eqref{eq:claim:ladder7.5}. If $\alpha < \theta_\gamma(\mu)$, then, by \eqref{eq:claim:ladder8}, $\alpha$ is a limit point of $D_\varrho(\theta_\gamma(\mu), p)$. Furthermore, since $D_\varrho(\theta_\gamma(\mu), p)$ is closed in $\theta_\gamma(\mu)$ by induction hypothesis, $\alpha$ belongs to $D_\varrho(\theta_\gamma(\mu), p)$ and, by \eqref{eq:claim:ladder7.5} again, we conclude that also in this case $\alpha \in D_\varrho(\gamma, p)$. Overall, $\alpha \in D_\varrho(\gamma, p)$, and thus $D_\varrho(\gamma, p)$ is a closed subset of $\gamma$.
\end{proof}
\end{proof}

\section{Open questions}\label{sec:conclusions}

\fxnote*{}{Many questions remain open. We list here a selection of them. We start with the main question stemming from Ditor's work, which remains open} \cite{ MR2768581, MR2926318}:

\begin{question}\label{question:1}
Is the existence of a $3$-ladder of cardinality $\aleph_2$ provable in $\mathsf{ZFC}$?
\end{question}

The next natural questions, in light of our results, are:

\begin{question}\label{question:2}
Is the existence of a (special) $(2, \aleph_1)$-ladder of cardinality $\aleph_2$ provable in $\mathsf{ZFC}$?
\end{question}

\begin{question}\label{question:4}
Does the existence of a $(2, \aleph_1)$-ladder of cardinality $\aleph_2$ imply the existence of a $3$-ladder of cardinality $\aleph_2$ and vice versa?
\end{question}

\fxnote*{}{It remains open whether Ditor's Problem~\ref{ditprob2} has a positive answer in $\mathsf{ZFC}$:

\begin{question}
Is the existence of a join-semilattice of breadth $2$ and cardinality $\aleph_{\omega+1}$ whose principal ideals have cardinality $< \aleph_\omega$ provable in $\mathsf{ZFC}$?
\end{question}}

Regarding Question~\ref{question:1}, it is somewhat unexpected for us that the existence of a 3-ladder of cardinality $\aleph_2$ can be derived from either $\square_{\aleph_1}$ (Theorem~\ref{thm:main3}) or from $\mathsf{MA}(\aleph_1)$ \cite{MR2609217}. It means that the existence of this structure follows from two axioms \fxnote*{}{that belong to two classes of axioms} usually considered, in some sense, ``orthogonal'', as already noted by Wehrung\footnote{Wehrung is referring to $(\omega_1,1)$-morasses and not to $\square_{\aleph_1}$, but the concept remains valid, as both morasses and squares can be viewed as incompactness principles.} \cite{MR2609217}: \fxnote*{}{indeed, $\square_{\aleph_1}$ is a paradigmatic incompactness principle at $\omega_2$, whereas strengthenings of $\mathsf{MA}(\aleph_1)$, like $\mathsf{PFA}$ and $\mathsf{MM}$, are known to enforce substantial compactness at $\omega_2$ and, in particular, imply the failure of $\square_{\aleph_1}$ \cite{MR776640,MR763902, MR2838054}.}

\printbibliography

\end{document}